\pdfoutput=1

\documentclass{article}
\usepackage{graphicx,amsfonts,amsthm, amsmath,amssymb,tikz,tikz-cd,mathrsfs,hyperref} 

\usepackage{wasysym}

\title{Wreath Macdonald polynomials, quiver varieties, and quasimap counts}
\author{Jeffrey Ayers and Hunter Dinkins}
\date{}

\theoremstyle{definition}
\newcommand{\core}{\text{core}}
\newcommand{\quot}{\text{quot}}
\newcommand{\tb}{\mathcal{V}}

\newcommand{\dv}{\mathsf{v}}
\newcommand{\dw}{\mathsf{w}}

\newcommand{\bT}{\mathsf{T}}
\newcommand{\bA}{\mathsf{A}}
\newcommand{\parts}{\mathcal{P}}
\newcommand{\C}{\mathbb{C}}
\newcommand{\Q}{\mathbb{Q}}
\newcommand{\desc}{\mathcal{D}}
\newcommand{\partitions}{\mathcal{P}}
\newcommand{\qta}{\mathscr{U}_{\hbar}(\ddot{\mathfrak{g}})}
\newcommand{\moalgebra}{\mathscr{U}^{\text{OS}}}
\newcommand{\uqgh}{\mathscr{U}_{\hbar}(\dot{\mathfrak{g}})}
\newcommand{\rmat}{\mathcal{R}}
\newcommand{\fock}{\mathsf{G}}
\newcommand{\ferfock}{\mathsf{F}}
\newcommand{\bos}{\mathsf{B}}
\newcommand{\qint}[2]{\left[#1\right]_{#2}}
\newcommand{\swap}{\mathsf{swap}}
\newcommand{\pneg}{\mathsf{neg}}
\newcommand{\inv}{\mathsf{inv}}
\newcommand{\extpow}{\bigwedge}
\newcommand{\canon}{\mathcal{K}}

\usepackage[colorinlistoftodos]{todonotes}

\usepackage{biblatex}
\newtheorem{thm}{Theorem}[section]
\newtheorem{cor}[thm]{Corollary}
\newtheorem{lem}[thm]{Lemma}
\newtheorem{prop}[thm]{Proposition}
\newtheorem{conj}[thm]{Conjecture}
\newtheorem{defn}[thm]{Definition}
\newtheorem{ex}[thm]{Example}

\newtheorem{rem}[thm]{Remark}
\newtheorem{ques}[thm]{Question}

\newcommand{\hilb}{\text{Hilb}}
\newcommand{\qv}{\mathcal{M}}
\newcommand{\stackqv}{\mathcal{M}^{\text{st}}}
\newcommand{\ch}{\text{char}}
\newcommand{\qm}{\text{QM}}

\newcommand{\stab}{\text{Stab}}
\newcommand{\qmpol}{\mathscr{T}^{1/2}}
\newcommand{\ns}{\text{ns }}
\newcommand{\rel}{\text{rel }}
\newcommand{\cver}[1]{\hat{V}^{\left(#1 \right)}}
\newcommand{\ver}[1]{V^{\left(#1 \right)}}
\newcommand{\ev}{\text{ev}}
\newcommand{\vrs}{\widehat{\mathcal{O}}_{\text{vir}}}
\newcommand{\Hom}{\text{Hom}}

\begin{document}
\maketitle

\begin{abstract}
    We study the $K$-theoretic enumerative geometry of cyclic Nakajima quiver varieties, with particular focus on $\text{Hilb}^{n}([\mathbb{C}^{2}/\mathbb{Z}_{l}])$, the equivariant Hilbert scheme of points on $\mathbb{C}^2$. The direct sum over $n$ of the equivariant $K$-theories of these varieties is known to be isomorphic to the ring symmetric functions in $l$ colors, with structure sheaves of torus fixed points identified with wreath Macdonald polynomials. Using properties of wreath Macdonald polynomials and the recent identification of the Maulik-Okounkov quantum affine algebra for cyclic quivers with the quantum toroidal algebras of type $A$, we derive an explicit formula for the generating function of capped vertex functions of $\text{Hilb}^{n}([\mathbb{C}^{2}/\mathbb{Z}_{l}])$ with descendants given by exterior powers of the $0$th tautological bundle. We also sharpen the large framing vanishing results of Okounkov, providing a class of descendants and cyclic quiver varieties for which the capped vertex functions are purely classical. Our results also suggest certain integrality and wall-crossing conjectures for capped vertex functions.

\end{abstract}
{\small\tableofcontents\par}

\setcounter{footnote}{0}
\renewcommand{\thefootnote}{\arabic{footnote}}

\section{Introduction}

\subsection{Enumerative geometry of quiver varieties}
The pioneering work of Maulik-Okounkov \cite{MO} demonstrated that the quantum cohomology of Nakajima quiver varieties is controlled by certain geometrically constructed quantum groups. The $K$-theoretic version of this work, relating quasimaps counts to quantum affine algebras, was developed in \cite{Ok17} and \cite{OkSm}. This interplay between geometry and representation theory has lead to a large body of work in areas like symmetric functions, integrable systems, characteristic classes, and $q$-difference equations, \cite{AOElliptic,AgOk, BRbows, NegGor, KorZeit2,indstab2,OkSm, MirSym2, MirSym1, qdehilb,zhu2}. On the one hand, $K$-theoretic counts present additional technical challenges compared to cohomological counts; on the other hand, they are required to observe additional structure, most importantly 3d mirror symmetry \cite{AOElliptic,AgOk, dinkms2, dinkms1, dinksmir2, dinksmir}. By the degeneration and glue formulas \cite{Ok17}, general $K$-theoretic quasimap counts in Nakajima quiver varieties are built from the following three objects: the (bare) vertex, the capping operator, and the capped vertex. Each corresponds to a different choice of ``boundary condition" for quasimaps. We touch briefly on the the first two before turning in more detail to the third, which is the main focus of this paper.

For many examples of interest, the bare vertex can be computed by localization, see \cite{AgOk, dinksmir2}. It provides a class of hypergeometric series generalizing the well-known $_{n+1} \phi _{n}$ basic hypergeometric series studied at length in, for example, \cite{hypergeo}. Vertex functions are known to satisfy certain $q$-difference equations with regular singularities. The monodromy of these equations was determined in \cite{AOElliptic} and \cite{indstab2}.

The capping operator is the fundamental solution of the quantum difference equation, which is the $K$-theoretic version of the quantum differential equation. For Nakajima quiver varieties, these equations generalize the qKZ and dynamical difference equations studied in \cite{FR, TV3, TV2}. The geometric approach to these equations was used by Aganagic-Okounkov in \cite{AgOk} to provide integral formulas for solutions.

The third fundamental curve count, the capped vertex, is the main object of interest for this paper. In some ways, the capped vertex is simpler than the first two curve counts: it is defined in non-localized $K$-theory and is purely classical for quiver varieties with large enough framing, a phenomenon known as ``large framing vanishing" and proved in \cite{Ok17}. It was proved by Smirnov \cite{Sm16} for the Jordan quiver and by Zhu \cite{zhu3} for general quivers (and earlier in cohomology in \cite{PPRationality}) that the capped vertex is a rational function of the quantum parameters (typically called K\"ahler parameters in this setting). On the other hand, computing the capped vertex explicitly in examples is difficult; the localization theorem essentially reduces the problem to computing both the bare vertex and capping operator. Building on \cite{Sm16}, the recent work of Smirnov and the first author \cite{AySm} developed explicit formulas for the capped vertex for what is arguably the most important quiver variety: the Hilbert scheme of points in the plane. Their formulas reveal a beautiful relationship with Macdonald polynomials and give an enumerative-geometric interpretation to certain evaluations of them.

In this paper, we push the methods of Ayers-Smirnov further, studying the capped vertex functions for cyclic quiver varieties. Our main theorem, Theorem \ref{mainthmintro} below, describes the capped vertex functions for certain cyclic quiver varieties in terms of wreath Macdonald polynomials \cite{BFwreath, Haiman1}. Our main theorem can be interpreted as describing capped vertex functions as certain evaluations of wreath Macdonald polynomials, see Theorem \ref{mainthm2intro}.

While replacing Macdonald polynomials with wreath Macdonald polynomials aligns with what is generally expected when moving from the Jordan quiver to cyclic quivers, the technical details of our proof rely heavily on the exact choice of cyclic quivers and descendants appearing in our main theorem. While some of our tools apply more generally than in Theorem \ref{mainthmintro} (for example our results on large framing vanishing in Theorem \ref{lfvintro}), we are unsure if similarly nice formulas exist for other capped vertex functions. We now turn to discuss in more detail the setup and results of this paper.

\subsection{The capped descendant vertex}

Let $\qv:=\hilb^{n}([\mathbb{C}^{2}/\Gamma_{l}])$ where $\Gamma_{l}=\mathbb{Z}/l\mathbb{Z}$, which acts on $\mathbb{C}^2$ by $\zeta \cdot (x,y)=(\zeta x, \zeta^{-1} y)$ where $\zeta$ is a primitive $l$th root of unity.. Concretely, the variety $\qv$ parameterizes $\Gamma_{l}$ invariant ideals $I \subset \mathbb{C}[x,y]$ such that $\ch_{\Gamma_{l}} \left(\mathbb{C}[x,y]/I\right)=\sum_{i=0}^{l-1} n \chi^{i}$, where $\chi^{i}$ are the irreducible characters of $\Gamma_{l}$. In particular, $\qv\subset \hilb^{nl}(\mathbb{C}^{2})$.

The variety $\qv$ is an example of a Nakajima quiver variety \cite{NakALE, NakQv} with dimension vector $\dv=(n,n,\ldots,n)$ and thus has a canonical presentation as a geometric invariant theory quotient
$$
\qv=\mu^{-1}(0)/\!\!/_{\theta} G_{\dv}, \quad G_{\dv}=\prod_{i=0}^{l-1} GL(n)
$$
This description is reviewed in section \ref{NQVsection}. Hence we can apply the general machinery of quasimaps to GIT quotients developed in \cite{qm}. To that end, let $\qm^{d}$ be the moduli space of stable quasimaps of degree $d \in \mathbb{Z}^{l}$ relative to the point $0 \in \mathbb{P}^{1}$. It compactifies the space of genuine maps from $\mathbb{P}^{1}$ to $\qv$.

By definition of $\qm^{d}$, there are evaluation maps
$$
\begin{tikzcd}
    & \qm^{d} \arrow{ld}{\ev_{0}} \arrow{rd}{\ev_{\infty}} & \\
    \qv & & \stackqv:=\left[ \mu^{-1}(0)/G_{\dv}\right]
\end{tikzcd}
$$
The notation $[\mu^{-1}(0)/G_{\dv}]$ in the bottom right corner of the diagram refers to the stack quotient. By definition of $\qv$, there is an inclusion $\iota: \qv \hookrightarrow \stackqv$. There is an action of a torus $\bT$ on each piece of the above diagram such that the evaluation maps are $\bT$-equivariant. 

Then the capped vertex function with descendant $\tau \in \desc:= K_{\bT}(\stackqv)$ is defined as

$$
\cver{\tau}= \canon^{-1/2} \sum_{d} \ev_{0,*}\left(\ev_{\infty}^{*}(\tau) \otimes \vrs^{d} \right) z^{d} \in K_{\bT}(\qv)[q^{\pm 1}][[z]]
$$
where $\vrs^{d}$ is the symmetrized virtual structure sheaf on $\qm^{d}$ and $\canon$ is the canonical bundle of $\qv$\footnote{The presence of $\canon^{-1/2}$ is simply a normalization choice and means that $\cver{\tau}=\iota^{*}\tau+O(z)$.}. The sum is taken over all degrees $d$ such that $\qm^{d}\neq \emptyset$ and $z^d$ is an element of the semi-group algebra on the cone of effective quasimap degrees. Since $\qm^{0}\cong \qv$, the expression $\cver{\tau}$ can be thought of as a $z$-deformation of the tautological class $\iota^{*}\tau$ in $K_{\bT}(\qv)$. The $z$-parameters are referred to as K\"ahler parameters and taking the constant term in $z$ recovers $\iota^{*}\tau$.

The series $\cver{\tau}$ was proved in \cite{Sm16} to be a rational function in $z$ for the rank $r$ instanton moduli space, for which the cyclic quiver at $l=1$ is the rank 1 version. An explicit formula for the $l=1$ case was given in \cite{AySm}. Quasimap counts in cyclic quiver varieties are related to curve counts in threefolds, see \cite{liu}. The cohomological version of this relationship goes back to Okounkov-Pandharipande \cite{OPQuanCoh} and Maulik-Oblomkov \cite{ObloMaul}.

Explicitly, the space of descendants is 
$$
\desc=K_{\bT}(\stackqv)=K_{\bT \times G_{\dv}}(pt)=K_{\bT}(pt) \otimes \mathbb{Z}[x_{i,j}^{\pm 1}]_{\substack{0 \leq i \leq l-1 \\ 1 \leq j \leq n}}^{\text{Sym}}
$$
where the superscript $\text{Sym}$ refers to Laurent polynomials that symmetric in the Chern roots $x_{i,1},\ldots, x_{i,n}$ for each $i$ separately. Equivalently, $\desc$ is the ring generated by the Schur functors of the tautological vector bundles $\tb_{i}$, $0 \leq i \leq l-1$, on $\qv$. Let $\tau_{i}:=\sum_{j=0}^{n} \bigwedge^{j} \tb_{i}$.

By equivariant localization, to understand $\cver{\tau_{i}}$ it suffices to study the restrictions $\cver{\tau_{i}}_{\lambda}:= \cver{\tau_{i}}|_{\lambda} \in K_{\bT}(pt)[q^{\pm 1}][[z]]$, where $\lambda \in \qv$ is a $\bT$-fixed point. As we will review in section \ref{NQVsection}, $\qv^{\bT}$ is in natural bijection with partitions $\lambda$ of $nl$ with empty $l$-core. Let $\partitions$ denote the set of all partitions.

Now we consider the varieties $\qv$ for all $n$ and denote $\qv(n):=\text{Hilb}^{n}([\mathbb{C}^2/\Gamma_{l}])$. By \cite{BFwreath}, there is an isomorphism
\begin{equation}\label{Ksymintro}
\bigoplus_{n \geq 0} K_{\bT}(\qv(n))_{loc} \cong \Lambda_{t_1,t_2}^{\otimes l}
\end{equation}
where $\Lambda_{t_1,t_2}=\mathbb{Q}(t_{1},t_{2})[p_1,p_2,\ldots]$ is the ring of symmetric functions\footnote{The usual $q$ and $t$ of Macdonald theory match our parameters by $q=t_{1}, t=t_{2}$.}. The ring $\Lambda_{t_1,t_2}^{\otimes l}$ is generated as an algebra over $\mathbb{Q}(t_1,t_2)$ by the ``colored" power sums: $p^{(i)}_{n}=1 \otimes \ldots \otimes p_{n} \otimes \ldots \otimes 1$ where $p_{n}$ is in the $i$th position. There is another interesting basis of $\Lambda_{t_1,t_2}^{\otimes l}$ given by the wreath Macdonald polynomials of Haiman \cite{Haiman1}, see also the survey of Orr-Shimozono \cite{OrrShim}. We will review the basics of wreath Macdonald theory in section \ref{symsec}. Under the isomorphism \eqref{Ksymintro}, the skyscraper sheaf of a fixed point $\lambda$ is mapped to the wreath Macdonald polynomial $H_{\lambda}$. The two spaces in \eqref{Ksymintro} are equipped with the action of the quantum toroidal algebra, compatible up to certain scalars. We assume in this paper, see section \ref{sec: qta}, that these scalars can be neglected.

Considering the generating function for $\cver{\tau_{0}}_{\lambda}$ for $\lambda \in \parts$ with $\text{core}_{l}(\lambda)=\emptyset$ reveals a beautiful structure. This is the content of our main theorem. For notations that have not yet been defined, see the main text.

\begin{thm}[Theorem \ref{mainthm}]\label{mainthmintro}
Let
$$
N_{\lambda}=\prod_{\substack{\square \in \lambda \\ h_{\lambda}(\square) \equiv 0 \mod l}} (1-t_1^{l(\square)+1}t_2^{-a(\square)})(1-t_1^{-l(\square)}t_2^{a(\square)+1})
$$
and
$$
C^{(i)}_{n}:=\sum_{\substack{j,k=0 \\ -i-j+k \equiv 0 \mod l}}^{l-1} t_{1}^{nj} t_{2}^{nk} 
$$
Then we have
     \begin{multline*}
\sum_{\substack{\lambda \\ \core_{l}(\lambda)=\emptyset}} N_{\lambda}^{-1} H_{\lambda} \cver{\tau_{0}}_{\lambda}  = \\
\exp\left(\sum_{i=0}^{l-1}\sum_{n \geq 1} \frac{C^{(i)}_{n} p^{(i)}_{n}}{n (1-t_{1}^{nl})(1-t_{2}^{nl})} \left(1-\frac{(-\hbar^{1/2} u)^{n} - (\hbar u z_{0} \ldots z_{l-1})^{n}}{(\hbar^{1/2})^{n}-(-z_{0} \ldots z_{l-1})^n} \right)   \right)
\end{multline*}
\end{thm}

Theorem \ref{mainthmintro} completely determines $\cver{\tau_{0}}_{\lambda}$ for each $\lambda$, as the single term $\cver{\tau_{0}}_{\lambda}$ can be extracted by taking the wreath inner product of both sides with $H_{\lambda}$. Of course, such an extraction does not lead to any explicit formulas beyond Theorem \ref{mainthmintro}.

Theorem \ref{mainthmintro} closely resembles the Cauchy identity for wreath Macdonald polynomials. This is not a coincidence; rather, we will use the wreath Cauchy identity in the proof. We can equivalently recast the theorem as giving an interpretation to a special evaluation of wreath Macdonald polynomials.

\begin{thm}\label{mainthm2intro}
Let $\lambda \in \partitions$ such that $\core_{l}(\lambda)=\emptyset$. Then
    $$
H_{\lambda}|_{p^{(j)}_{n}=\delta_{0,j}\left(1-\frac{(-\hbar^{1/2} u)^{n} - (\hbar u z_{0} \ldots z_{l-1})^{n}}{(\hbar^{1/2})^{n}-(-z_{0} \ldots z_{l-1})^n} \right) } = \cver{\tau_0}_{\lambda}
    $$
\end{thm}

Setting $z_{i}=0$ in Theorem \ref{mainthm2intro}, we recover the evaluation formula for wreath Macdonald polynomials, see Proposition \ref{prop: classicaleval}. When the core is empty, this states
$$
    H_{\lambda}|_{p^{(j)}_{k} =\delta_{0,j} - \delta_{m,j} (-u)^{k}}  = \prod_{\substack{\square=(a,b) \in \lambda \\ a-b \equiv m \mod l}}(1+u t_1^{b-1} t_2^{a-1})
    $$
    This formula, and its generalization to arbitrary core, was recently proved by Romero and Wen in \cite{WenRomero}. 

We are unsure if the evaluation formulas for nonempty core or even the empty core evaluation for $m \neq 0$ admit ``quantum deformations" as in Theorem \ref{mainthm2intro}.

We deduce a few corollaries from Theorem \ref{mainthmintro}. 

\begin{cor}
    Viewed as a rational function of $z$, the poles of $\cver{\extpow^{j}\tb_{0}}$ are contained in the locus $\{(-z_0\ldots z_{l-1} \hbar^{-1/2})^k=1 \, \mid \, 1 \leq k \leq j\}$.
\end{cor}

Specializing $z_0\ldots z_{l-1}=-\hbar^{-1/2}$ yields the following.

\begin{cor}
We have the evaluation
    \[
\cver{\tau_{0}}|_{z_0\ldots z_{l-1}=-\hbar^{-1/2}}=1
    \]
    Equivalently,
    \[
\cver{\extpow^{j}\tb_{0}}|_{z_0\ldots z_{l-1}=-\hbar^{-1/2}}=0 \quad \text{for $j>0$}
    \]
\end{cor}

Consider the following automorphisms on the ring of symmetric functions
\begin{align*}
    \swap&: t_1 \mapsto t_2, \,  t_2 \mapsto t_1 \\
    \pneg&: p^{(i)}_{n} \mapsto p^{(-i)}_{n}
\end{align*}
where the superscripts are taken modulo $l$. It is known that $\swap \, \pneg \, H_{\lambda}=H_{\lambda^{t}}$ when $\core_{l}(\lambda)=\emptyset$, see \cite{OrrShim} Proposition 3.21. One readily checks that $\swap \,  C^{(i)}_{n}=C^{(-i)}_{n}$ and $\swap \, N_{\lambda}=N_{\lambda^{t}}$. So applying $\swap \, \pneg$ to Theorem \ref{mainthmintro} gives the following.

\begin{cor}
    \[
\swap \,
\cver{\tau_{0}}_{\lambda}=\cver{\tau_{0}}_{\lambda^{t}}
    \]
\end{cor}




\subsection{Ingredients in the proof}

Let us now discuss the ingredients in the proof of Theorem \ref{mainthmintro}.

Viewing $\qv(n)$ as a Nakajima quiver variety not only allows one to define quasimap moduli spaces but also opens the door to the representation-theoretic machinery developed in \cite{Ok17} and \cite{OkSm}. In particular, \cite{Ok17} and \cite{OkSm} construct an action of a Hopf algebra on the $K$-theory of quiver varieties and prove that the capping operator is determined by $q$-difference equations involving operators in this Hopf algebra. The important recent work of Zhu \cite{zhu2} identifies this Hopf algebra with the quantum toroidal algebra of type $A$, confirming a longstanding expectation. The cohomological version of this identification was proved earlier by Botta-Davison \cite{BD} and Schiffman-Vasserot \cite{SV23} independently. By Zhu's theorem, we are able to use known results about quantum toroidal algebras and their slope subalgebras to approach capped vertex functions. Studying the $R$-matrix of the slope zero subalgebra in the tensor product of two Fock representations is what leads us to restrict our attention to $\tau_{0}$ rather than a general $\tau_{m}$.

The quiver variety perspective on $\qv(n)$ is again useful, as one can vary the ``framing" vector. Viewed as a Nakajima quiver variety, the dimension vector for $\qv(n)$ is $\dv=(n,n,\ldots,n)$ and the framing vector is $\dw=(1,0,\ldots,0)=:\delta_{0}$. As is well-known, see \cite{Naktens}, one can embed any quiver variety as a $\mathbb{C}^{\times}$-fixed component of another quiver variety with larger framing. The gauge group $G_{\dv}$, and hence the space of descendants $\desc$, is left unchanged. Okounkov proves that for any given $\tau$, there exists a sufficiently large framing vector such that the quantum corrections of $\cver{\tau}$ vanish, see \cite{Ok17} section 7. We refine Okounkov's result here, proving the following.
\begin{thm}[Theorem \ref{lfv}]\label{lfvintro}
    Let $\dv=(n,n,\ldots,n)$ and $\dw=\delta_{0}+\delta_{m}$ and let $\cver{\tau_{m}}$ be the capped vertex function for the quiver variety $\qv(\dv,\dw)$ with descendant $\tau_{m}$. Then the quantum corrections for $\cver{\tau_{m}}$ vanish, i.e.,
    $$
\cver{\tau_{m}}=\tau_{m}
$$
\end{thm}

Given that our large framing vanishing result (Theorem \ref{lfv}) holds for $\tau_{m}$ for any $m$ and that Proposition \ref{prop: classicaleval} likewise holds for any $m$, it is in our view not unreasonable to hope for a nice formula for $\cver{\tau_{m}}$ for $m \neq 0$. Nevertheless, as we have said above, we do not know how to generalize the $R$-matrix formulas to apply to these cases.

In the remainder of the introduction, we will discuss a few more applications and conjectures arising from Theorem \ref{mainthmintro}.

\subsection{Calabi-Yau specialization}

As we will review in Theorem \ref{VPsiV}, one can write the capped vertex as a product of the capping operator and the bare vertex:
\begin{equation}\label{VpsiVintro}
\cver{\tau}= \Psi \cdot \ver{\tau}
\end{equation}
Just as for $\cver{\tau}$, the two terms on the right side are $K$-theory valued power series in the $z$ variables. One can consider the so-called \emph{Calabi-Yau limit}, which, in our notation, corresponds to setting $q=\hbar$. For an arbitrary descendant $\tau$, $\cver{\tau}$ is Laurent polynomial in $q$ and this specialization loses information. Remarkably, the formula in Theorem \ref{mainthmintro} is manifestly independent of $q$. So one can specialize each term of \eqref{VpsiVintro} to obtain a much simpler equation for the capped vertex:
\[
\cver{\tau_{0}}=\Psi|_{q=\hbar}  \cdot \ver{\tau_{0}}|_{q=\hbar}
\]
The specialization $\ver{\tau_{0}}|_{q=\hbar}$ is quite extreme and makes contact with the representation theory of quantized Coulomb branches by the ``quantum Hikita conjecture" of \cite{KMP}. The descendant vertex $\ver{\tau_{0}}_{\lambda}|_{q=\hbar}$ is expected to coincide with the graded trace of the element $\tau_{0}$, viewed as an element of the quantized Coulomb branch, over the Verma module corresponding to $\lambda$, see \cite{DKK} where this is proved for ADE quiver varieties with minuscule framings. Furthermore, $\ver{\tau_{0}}_{\lambda}|_{q=\hbar}$ is a rational function of $z$ in contrast to $\ver{\tau_{0}}_{\lambda}$ which is $q$-hypergeometric.

The $q=\hbar$ specialization of $\Psi$ must similarly be drastic. Computer experiments show that it actually truncates to a polynomial in $z$. We conjecture that this holds in general.

\begin{conj}\label{Psipoly}
Let $X$ be a quiver variety and let $\Psi$ be the capping operator of $X$. The specialization
    \[
\Psi|_{q=\hbar}
    \]
    is a polynomial in $z$.
\end{conj}

As a consequence of Conjecture \ref{Psipoly}, we have the following.

\begin{prop}
    Assume Conjecture \ref{Psipoly} holds. Let $\tau$ be such that $\cver{\tau}$ is independent of $q$. Then the poles of $\cver{\tau}$, viewed as a rational function of $z$, are contained in the locus of poles of $\ver{\tau}|_{q=\hbar}$.
\end{prop}

We emphasize that computing $\ver{\tau}|_{q=\hbar}$ as a rational function of $z$--and thus computing its poles--is relatively straightforward. In fact, it reduces to the computation for zero-dimensional quiver varieties which were studied in depth in type $A$ in \cite{dinksmir3, dinksmir2} and in type $D$ in \cite{dinkjang}. When $r=1$ and $\tau=1$, the poles of $\ver{\tau}_{\lambda}|_{q=\hbar}$ can be seen to lie at $k$th roots of unity where $k$ runs over the set of hook lengths of $\lambda$, in agreement with Conjecture 3 of \cite{PPRationality}. As in \cite{dinksmir3}, other $\tau$ can be studied by applying some Macdonald difference operators.

\subsection{Global integrality}

Theorem \ref{mainthmintro} also suggests a new integrality property of capped vertex functions. Let us consider the example of $\hilb^{2}([\mathbb{C}^{2}/\Gamma_{2}])$. As a Nakajima variety, it is constructed using the dimension data $\dv=(2,2)$ and $\dw=(1,0)$ and the stability parameter $\theta=(1,1)$. There are five torus fixed points, labeled by the partitions of $4$. Theorem \ref{mainthmintro} gives the following formulas:
\begin{align*}
    &\cver{\tau_{0}}_{(4)}=1+\frac{\sqrt{\hbar} (\sqrt{\hbar} z_0 z_1 +1) (1+t_{1}^{-2}) }{(z_0 z_1+\sqrt{\hbar})} u +\frac{\hbar (\sqrt{\hbar} z_0 z_1 +1) (-\hbar z_0 z_1 +z_0 z_1 +\sqrt{\hbar} ( z_0^2 z_1^2-1) t_{1}^{-2})}{(z_0 z_1+\sqrt{\hbar})(z_0^2 z_1^2-\hbar)} u^2 \\
   & \cver{\tau_{0}}_{(3,1)}= \cver{\tau_{0}}_{(4)}\\
   & \cver{\tau_{0}}_{(2,2)}=1+\frac{\sqrt{\hbar} (\sqrt{\hbar} z_0 z_1 +1) (1+\hbar^{-1}) }{(z_0 z_1+\sqrt{\hbar})} u +\frac{\hbar (\sqrt{\hbar} z_0 z_1 +1) (-\hbar z_0 z_1 +z_0 z_1 +\sqrt{\hbar} ( z_0^2 z_1^2-1) \hbar^{-1})}{(z_0 z_1+\sqrt{\hbar})(z_0^2 z_1^2-\hbar)} u^2  \\
   & \cver{\tau_{0}}_{(2,1,1)}= 1+\frac{\sqrt{\hbar} (\sqrt{\hbar} z_0 z_1 +1) (1+t_{2}^{-2}) }{(z_0 z_1+\sqrt{\hbar})} u +\frac{\hbar (\sqrt{\hbar} z_0 z_1 +1) (-\hbar z_0 z_1 +z_0 z_1 +\sqrt{\hbar} ( z_0^2 z_1^2-1) t_{2}^{-2})}{(z_0 z_1+\sqrt{\hbar})(z_0^2 z_1^2-\hbar)} u^2\\
  &  \cver{\tau_{0}}_{(1,1,1,1)}= \cver{\tau_{0}}_{(2,1,1)}
\end{align*}
By the localization theorem, this uniquely determines $\cver{\tau_{0}}$. A moment's inspection reveals that 
\begin{multline}\label{globalcver}
\cver{\tau_{0}}=1+\frac{\sqrt{\hbar} (\sqrt{\hbar} z_0 z_1 +1) \tb_{0} }{(z_0 z_1 +\sqrt{\hbar})} u \\ +\frac{\hbar (\sqrt{\hbar} z_0 z_1  +1) (-\hbar z_0 z_1  +z_0 z_1  +\sqrt{\hbar} ( z_0^2 z_1^2-1) \extpow^{2}\tb_{0} )}{(z_0 z_1 +\sqrt{\hbar})(z_0^2 z_1 ^2-\hbar)} u^2
\end{multline}
In other words, not only is $\cver{\tau_{0}}$ a rational function of $z$, it is an element of \emph{non-localized} $K$-theory over the field $\mathbb{Q}(t_1^{1/2},t_2^{1/2},z)$. A-priori, it is only the coefficients of $z$ in the series defining $\cver{\tau}$ which lie in integral $K$-theory, and this does not guarantee the global property. We refer to this property as the ``global integrality" of the capped vertex. We conjecture that it holds in general.

\begin{conj}\label{conj: global integrality}
    The capped vertex function $\cver{\tau}$ for any quiver variety $X$ and any descendant $\tau$ is globally integral. More precisely, for any descendant $\tau$, there exists a polynomial $d_{X,\tau}(z) \in K_{\bT}(pt)[z]$ such that $d_{X,\tau}(z) \cver{\tau} \in K_{\bT}(X)[z]$.
\end{conj}

One more interesting observation from \eqref{globalcver} is that $\cver{\extpow^{k}\tb{0}}$, as a $K$-theory class, lies in the span of $1$ and $\extpow^{k}\tb{_0}$. We do not know if this holds for arbitrary descendants.

\subsection{Wall-crossing}

It is expected in enumerative geometry that invariants of spaces related by a change in stability parameters should be related by some ``wall-crossing formulas". This has been studied in various settings, see for example \cite{BCY}, \cite{CR}, \cite{joyce}, \cite{moreira}, \cite{PSW}. In the setting of this paper--$K$-theoretic invariants of quasimaps to GIT quotients--we do not know of a precise conjecture. It is generally expected, see for example \cite{BCR}, that when a generating series like $\cver{\tau}$ is a rational function of $z$, it should be invariant under changes of stability, up to a monomial transformation of the $z$-variables. Different choices of stability should correspond to power series expansions around different points.

The formula of Theorem \ref{mainthmintro} provides some insight into the sort of wall-crossing that should hold here. To explain why, let us again consider the example of $\hilb^{2}([\mathbb{C}^{2}/\Gamma_{2}])$.

Specializing $z=0$, in \eqref{globalcver} one obtains $\tau_0$. Equivalently, specializing $\cver{\tau_{0}}|_{\lambda}$, one obtains $\tau_{0}|_{\lambda}$. One way to probe what other stability conditions $\cver{\tau_{0}}$ ``knows about" is by taking other limits in $z$. For example,
\[
\lim_{z_1,z_2 \to \infty} \cver{\tau_{0}}_{(4)}=1+(t_1 t_2 + t_1^{-1} t_2) u +t_2^2 u^2
\]
which is easily computed to be the character of $\tb_{0}$ at the appropriate fixed point on the quiver variety with $\dv=(2,2)$, $\dw=(1,0)$, and the \emph{opposite stability condition} $\theta=(-1,-1)$.

If instead one uses the stability condition $\theta'=(-1+\epsilon,1)$ for $0<\epsilon\ll 1$, one obtains $\hilb^{2}(\mathcal{A}_{1})$, the Hilbert scheme of points on the resolution $\mathcal{A}_{1}$ of the type $A_1$ surface singularity. The fixed points are in bijection with partitions of $4$, and the characters of the exterior powers of the $0$th tautological bundle on this space are
\begin{align*}
  &  \tau_{0}'|_{(4)}=1+(1+t_1^{-2})u+t_1^{-2} u^2 \\
   & \tau_{0}'|_{(3,1)}=1+(1+t_1 t_2^{-1})u+t_1 t_2^{-1} u^2 \\
   & \tau_{0}'|_{(2,2)}=1+2 u+u^2 \\
  & \tau_{0}'|_{(2,1,1)}=1+(1+t_1^{-1} t_2)u+t_1^{-1} t_2 u^2  \\
  & \tau_{0}'|_{(1,1,1,1)}=1+(1+t_2^{-2})u+t_2^{-2} u^2
\end{align*}
These can be computed by writing down explicit $\bT$-fixed $\theta'$-stable quiver representations. A more powerful method is to apply some composition of reflection isomorphisms $R_{i}$ of \cite{Nakrefl}, to $\hilb^{2}(\mathcal{A}_{1})$ until reaching a quiver variety with positive stability condition, for which the characters of tautological bundles at fixed points are known. Combinatorial versions $R_{i}^{*}$ of reflection isomorphisms, which are operators on $\bT$-characters, are known to control the change in tautological bundle characters under reflection isomorphisms, see Theorem 4.3 of \cite{OrrShim}. For the quiver variety $\hilb^{2}(\mathcal{A}_{1})$, the application of a single reflection isomorphism at vertex 0 gives the quiver variety with $\dv=(3,2)$, $\dw=(1,0)$, and $\theta=(1,1)$. Either way, we obtain the formulas above.

If capped vertex functions for each fixed point were invariant under change of stability parameter, then there must be some (possibly infinite) solution for $(z_1,z_2)$ in the equation $\cver{\tau_{0}}_{(3,1)}=\tau_{0}'|_{(3,1)}$. It is straightforward to check that this equation has no solution. So capped vertex functions cannot be invariant under change in stability in this way.

For the quiver variety $\hilb^{2}(\mathcal{A}_{1})$, we have calculated $\cver{\tau}$ directly in the fixed point basis using a computer and formula \eqref{VpsiVintro}. In this case, torus fixed points on the moduli space giving rise to the bare vertex function are no longer isolated; so it is not clear how to generalize our proof of Theorem \ref{lfvintro} to this setting. Nevertheless, we record the computation:
\begin{align*}
  & \cver{\tau_{0}}_{(4)}=1+\frac{\sqrt{\hbar}(\sqrt{\hbar}z_1+1)(1+t_1^{-2})}{(z_1+\sqrt{\hbar})}u+\frac{\hbar(\sqrt{\hbar} z_1+1)(-\hbar z_1 +z_1 + \sqrt{\hbar}(z_1^2-1) t_1^{-2})}{(z_1+\sqrt{\hbar})(z_1^2-\hbar)}u^2 \\
    &  \cver{\tau_{0}}_{(3,1)}=1+\frac{\sqrt{\hbar}(\sqrt{\hbar}z_1+1)(1+t_1^{-1} t_2)}{(z_1+\sqrt{\hbar})}u+\frac{\hbar(\sqrt{\hbar} z_1+1)(-\hbar z_1 +z_1 + \sqrt{\hbar}(z_1^2-1) t_1^{-1} t_2)}{(z_1+\sqrt{\hbar})(z_1^2-\hbar)}u^2 \\
     &    \cver{\tau_{0}}_{(2,2)}=1+2 \frac{\sqrt{\hbar}(\sqrt{\hbar}z_1+1)}{(z_1+\sqrt{\hbar})}u+\frac{\hbar(\sqrt{\hbar} z_1+1)(-\hbar z_1 +z_1 + \sqrt{\hbar}(z_1^2-1))}{(z_1+\sqrt{\hbar})(z_1^2-\hbar)}u^2 \\
      &      \cver{\tau_{0}}_{(2,1,1)}=1+\frac{\sqrt{\hbar}(\sqrt{\hbar}z_1+1)(1+t_1 t_2^{-1})}{(z_1+\sqrt{\hbar})}u+\frac{\hbar(\sqrt{\hbar} z_1+1)(-\hbar z_1 +z_1 + \sqrt{\hbar}(z_1^2-1)t_1 t_2^{-1})}{(z_1+\sqrt{\hbar})(z_1^2-\hbar)}u^2 \\
   & \cver{\tau_{0}}_{(1,1,1,1)}=1+\frac{\sqrt{\hbar}(\sqrt{\hbar}z_1+1)(1+t_2^{-2})}{(z_1+\sqrt{\hbar})}u+\frac{\hbar(\sqrt{\hbar} z_1+1)(-\hbar z_1 +z_1 + \sqrt{\hbar}(z_1^2-1) t_2^{-2})}{(z_1+\sqrt{\hbar})(z_1^2-\hbar)}u^2
\end{align*}
As in \eqref{globalcver}, we write this as an integral $K$-theory class:
\[
\cver{\tau_{0}}=1+\frac{\sqrt{\hbar} (\sqrt{\hbar} z_1 +1) \tb_{0} }{(z_1+\sqrt{\hbar})} u +\frac{\hbar (\sqrt{\hbar} z_1 +1) (-\hbar z_1 +z_1 +\sqrt{\hbar} ( z_1^2-1) \extpow^{2}\tb_{0} )}{(z_1+\sqrt{\hbar})(z_1^2-\hbar)} u^2
\]
After the change of variables\footnote{Although we only need the change of $z_1$, this reflects the change of coordinates moving the \emph{ample cone} for $(-1+\epsilon,1)$ to the ample cone for $(1,1)$.} $z_0\mapsto z_0^{-1}, z_1\mapsto z_0 z_1$ this coincides with the class \eqref{globalcver}.

\begin{conj}
   Assume Conjecture \ref{conj: global integrality}. Let $X$ and $X'$ be quiver varieties constructed using the same $\dv$ and $\dw$ but different generic stability parameters. Let $\tb_i$ and $\tb_i'$ be the tautological bundles on $X$ and $X'$ respectively. Let $\widetilde{\Phi}$ be the map $\tb_i \mapsto \tb_{i}'$ and let $\Phi:K_{\bT}(X) \to K_{\bT}(X')$ be the induced map. Let $\tau$ be a descendant for $X$ and let $\tau'=\widetilde{\Phi}(\tau)$. Then the capped vertex functions for $X$ and $X'$ are related by
   \[
   \Phi(\hat{V}^{X,(\tau)})=\hat{V}^{X',(\tau')}
   \]
   where the K\"ahler variables are identified as $z_i'=\prod_{j} z_j^{a_{i,j}}$ for some $a_{i,j} \in \mathbb{Z}$.
\end{conj}

\subsection{Reflection isomorphisms}

In addition, our formulas give a hint about another possible wall-crossing result unique to the setting of Nakajima quiver varieties. Although we were not able to find a general formula like Theorem \ref{mainthmintro}, we are able to compute small examples of capped vertex functions for other type $A$ quiver varieties corresponding to nonempty $l$-cores using \eqref{VpsiVintro}. For the case of $\dv=(3,2)$, $\dw=(1,0)$, and $\theta=(1,1)$, the fixed points are indexed by partitions $\lambda$ of $5$ such that $\core_{2}(\lambda)=(1)$. We have computed the restriction to the fixed point $\lambda=(3,2)$ of the capped vertex with descendants $\extpow^1 \tb_{0}$ and $\extpow^{1} \tb_{1}$ respectively to be 
\begin{align*}
   &F_{0}(z_1,z_2)= \frac{(1-z_1z_2)(t_1+t_2)(t_1^2 t_2 z_1-t_1 z_1 - t_2 z_1 +t_2)}{(z_1-\hbar)(z_1 z_2-\hbar)} +\frac{\hbar (z_1-1)}{z_1-\hbar}\\
  & F_{1}(z_1,z_2)=\frac{(t_1+t_2)(z_1 z_2-1)}{z_1 z_2 -\hbar}
\end{align*}

The second summand in $F_0$ is exactly the capped vertex with descendant $\extpow^{1} \tb_{0}$ for the quiver variety with $\dv=(1,0)$, $\dw=(1,0)$, and $\theta=(1,1)$, which should be viewed as corresponding to the nontrivial $2$-core of the partition $(3,2)$. Remarkably, the system of equations
\begin{align*}
   & F_0(z_1,z_2)=1+t_1 t_2^{-1} \\
    & F_1(z_1,z_2)=t_1^{-1}+t_2^{-1}
\end{align*}
has a unique solution given by $(z_1,z_2)=(1,0)$. The $\bT$-characters on the right hand side are the restrictions of $\tb_{0}$ and $\tb_{1}$ on $\hilb^{2}(\mathcal{A}_{1})$ to the fixed point given by $(3,1)$. In fact, we have checked that this likewise works for any fixed point. This suggests that Nakajima's reflection isomorphisms may be related to quasimap counts in quiver varieties, providing another type of ``wall-crossing" besides variation of GIT parameter. We record this hope in the following question.

\begin{ques}
    Can the combinatorial reflection operators $R_{i}^{*}$ of Theorem 4.3 of \cite{OrrShim} admit ``quantum deformations" $R_{i}^{*}(z)$ which control how capped vertex functions change under reflection isomorphisms? Will the supposed operators $R_{i}^{*}(z)$ satisfy some dynamical relations?
\end{ques}

\subsection{Acknowledgements}

We would like to thank Davesh Maulik, Andrey Smirnov, Joshua Wen, and Tianqing Zhu for helpful discussions regarding this work. We are also grateful to an anonymous referee for suggestions improving this paper. J. Ayers was supported by the UNC Harold J. Glass Summer Research Fellowship. H. Dinkins was supported by NSF grant DMS-2303286 at MIT and the NSF RTG grant Algebraic Geometry and Representation Theory at Northeastern University DMS–1645877.

\section{Partitions and symmetric functions}\label{symsec}

We begin with a discussion of some combinatorial tools--partitions, symmetric functions, and wreath Macdonald polynomials--that will be used throughout this paper.

\subsection{Partitions and normalized Macdonald polynomials}\label{partitions}
A length $n$ partition of $k$ is a tuple $\lambda=(\lambda_1\geq\ldots\geq\lambda_n)$ with $|\lambda|:=\sum_i \lambda_i=k$. We conflate a partition $\lambda$ with its Young diagram, which is the set 
$$
\{(a,b) \in \mathbb{Z}^{2} \, \mid \, 1 \leq a \leq n, 1 \leq b \leq \lambda_{a}\}
$$
We call points in the Young diagram ``boxes" in $\lambda$ and we will sometimes denote a box of $\lambda$ by $(a,b)=\square \in \lambda$. Let $\parts$ denote the set of partitions.

The arm (resp. leg) of a box $\square$, written $a(\square)$ (resp. $l(\square)$) is the number of boxes of $\lambda$ strictly right (resp. above) of $\square$. The hook length of a box $\square \in \lambda$ is defined to be $h(\square):=a(\square)+l(\square)+1$. The content of a box $\square=(a,b)$ is $c(\square)=a-b$. The transpose of a partition $\lambda$ will be denoted by $\lambda'$.

\begin{ex} By our conventions the following Young Diagram represents the partition $(3,3,2,1)$
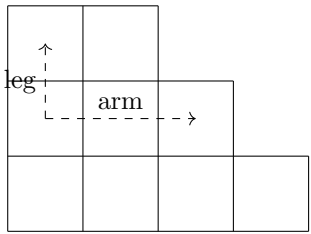
\begin{center}
    {\begin{tikzpicture}[scale=1]
\draw (0,0)--(4,0);;
\draw (0,1)--(4,1);;
\draw (0,2)--(3,2);;
\draw (0,3)--(2,3);;
\draw (0,0)--(0,3);;
\draw (1,0)--(1,3);;
\draw (2,0)--(2,3);;
\draw (3,0)--(3,2);;
\draw (4,0)--(4,1);;
\draw[dashed,->](0.5,1.5) --  node[above] {arm}  (2.5,1.5);
\draw[dashed,->](0.5,1.5) --  (0.5,2.5);
\node[] at (0.25,2.2){leg};
\end{tikzpicture}}
\end{center}
If $\square=(1,2)$, then $a(\square)=2$, $l(\square)=1$, $h(\square)=4$, and $c(\square)=-1$.
\end{ex}

\subsection{Symmetric functions}

Let 
$$
\Lambda_{t_1,t_2}:=\mathbb{Q}(t_1,t_2)[p_1,p_2,\ldots]
$$
be the ring of symmetric functions in infinitely many variables $x_1,x_2,\ldots$ over $\mathbb{Q}(t_1,t_2)$. The elements $p_{n}$ are the power sums, i.e. $p_{n}=\sum_{i \geq 1} x_{i}^{n}$. When there are several sets of variables around, we will instead write $p_{n}[x]$. The Schur functions $s_{\lambda}$ for $\lambda \in \parts$ form a basis for $\Lambda_{t_{1},t_{2}}$ over $\mathbb{Q}(t_{1},t_{2})$.

\subsection{Maya diagrams, cores, and quotients}
\begin{defn}
    A Maya diagram is a map $m: \mathbb{Z}\rightarrow \{0,1\}$ such that $m(i)=-1$ for $i>\!\!>0$ and $m(i)=1$ for $i<\!\!<0$.
\end{defn}
Pictorially we can view such a diagram as a chain of black and white beads along a line, where white beads correspond to integers $i$ such that $m(i)=-1$ and black beads to integers such that $m(i)=1$.

For $c\in \mathbb{Z}$ the vacuum diagram $v_c$ is defined as the Maya diagram where $v_c(i)=-1$ for $i\geq c$ and $v_c(i)=1$ for $i< c$. The charge of a Maya diagram  is defined as $$c(m):=|\{k\in \mathbb{Z}: k<0, m(k)=-1\}|-|\{k\in \mathbb{Z}: k\geq 0, m(k)=1\}|$$
Given any Maya diagram $m$, the Maya diagram $\widetilde{m}$ defined by
$$
\widetilde{m}(k)=m(k-c(m))
$$
always has charge $0$.

There is a well-known bijection between partitions and charge zero Maya diagrams, see \cite{James}. Given a partition $\lambda$, the corresponding Maya diagram is
$$
m_{\lambda}(i)=
\begin{cases}
  1 & i < 0 \text{ and } c^{\lambda}_{i}-c^{\lambda}_{i+1}=0 \\
  -1 & i < 0 \text{ and }  c^{\lambda}_{i} - c^{\lambda}_{i+1}=-1 \\
   1  & i \geq 0 \text{ and }  c^{\lambda}_{i}-c^{\lambda}_{i+1}=1 \\
   -1 & i\geq 0 \text{ and } 
 c^{\lambda}_{i} -c^{\lambda}_{i+1}=0
\end{cases}
$$
where $c^{\lambda}_{j}$ is the number of boxes of $\lambda$ with content $j$. See Figure \ref{mayaex} for an example.

For the rest of this subsection, we fix an integer $l \geq 0$. An $l$-strip in a partition $\lambda$ is a size $l$ connected subset of the outermost boxes of $\lambda$ whose removal gives another Young diagram. The $l$-core of $\lambda$, written $\core_{l}(\lambda)$ (or just $\core(\lambda)$ if the choice of $l$ is clear) is the partition obtained by successively removing all possible $l$-strips from $\lambda$. It does not depend on the order in which $l$-strips are removed. A partition is called an $l$-core partition if it does not have any $l$-strips.

\begin{prop}[\cite{Gordon} section 7.5, see also \cite{OrrShim} section 3.1.6]\label{coreroot}
    Let $\alpha_{i}$ for $1 \leq i \leq l-1$ be the simple roots of the $A_{l-1}$ root system. Let $\alpha_{0}=-\sum_{i=1}^{l-1} \alpha_{i}$. The assignment
    $$
\lambda \mapsto \sum_{(a,b) \in \lambda} \alpha_{a-b}
    $$
    defines a bijection between $l$-core partitions and the type $A_{l-1}$ root lattice. The subscripts in $\alpha_{a-b}$ are taken modulo $l$.    
\end{prop}

The $l$-quotient of $\lambda$, written $\quot_{l}(\lambda)$ is an $l$-tuple of partitions encoding the way in which $l$-strips are layered onto the $l$-core of $\lambda$. More precisely, given $\lambda$, let $m^{(i)}$ be the Maya diagram such that
$$
m^{(i)}(k)=m_{\lambda}(i+kl)
$$
Let $\lambda^{(i)}$ be the partition whose Maya diagram is $\widetilde{m^{(i)}}$ (the charge 0 shift of $m^{(i)}$). Then 
$$
\quot_{l}(\lambda)=(\lambda^{(0)},\lambda^{(1)},\ldots,\lambda^{(l-1)})
$$
It is known that the map
\begin{align}
    \nonumber \parts &\to \parts \times \parts^{l} \\ 
    \lambda &\mapsto (\core_{l}(\lambda),\quot_{l}(\lambda)) \label{quotcore}
\end{align}
is a bijection \cite{James}.

\begin{ex}
    Consider $\lambda=(3,2,2,1,1,1)$. Then $\core_{3}(\lambda)=(3,1)$ and $\quot_{3}(\lambda)=(\emptyset,\emptyset,(1,1))$. 

    \begin{figure}[!ht]
    \centering
\begin{tikzpicture}[circ/.style={shape=circle,draw,inner sep=1.5pt}]
    \node[on grid] (0) at (-1/2,0){};
   
\draw[-] (0,0) edge (-3,3) (1,1) edge (-2,4) (2,2) edge (0,4) (3,3) edge (1,5) (4,4) edge (3,5) (5,5) edge (4,6) (6,6) edge (5,7);

\draw[-] (0,0) edge (6,6) (-1,1) edge (5,7) (-2,2) edge (1,5) (-3,3) edge (-2,4);
    
\node[circ, on grid, below = of 0,label=below:$-1$](-1/2){};
\node[circ, fill=black,on grid, left= of -1/2,label=below:$-2$](-3/2){};
\node[circ, on grid, left= of -3/2,label=below:$-3$](-5/2){};
\node[circ, on grid, right= of -1/2,label=below:$0$](1/2){};
\node[circ, fill=black, on grid, right= of 1/2,label=below:$1$](3/2){};
\node[circ, on grid, right= of 3/2,label=below:$2$](5/2){};
\node[circ, on grid, right= of 5/2,label=below:$3$](7/2){};
\node[circ,on grid, right= of 7/2,label=below:$4$](9/2){};
\node[circ, fill=black,on grid, right= of 9/2,label=below:$5$](11/2){};

\draw[-,dotted] (-5/2) edge (-5/2,7/2) (-3/2) edge (-3/2,7/2) (-1/2) edge (-1/2,7/2) (1/2) edge (1/2,9/2) (3/2) edge (3/2,9/2) (5/2) edge (5/2,9/2) (7/2) edge (7/2,11/2) (9/2) edge (9/2,13/2) (11/2) edge (11/2,13/2);

\draw[fill=black,opacity=0.1] (0,0)--(-3,3)--(-2,4)--(0,2)--(1,3)--(2,2)--(0,0);

\end{tikzpicture}
 \caption{The Maya diagram for the partition $(3,2,2,1,1,1)$. The $3$-core is shaded.}
    \label{mayaex}
\end{figure}
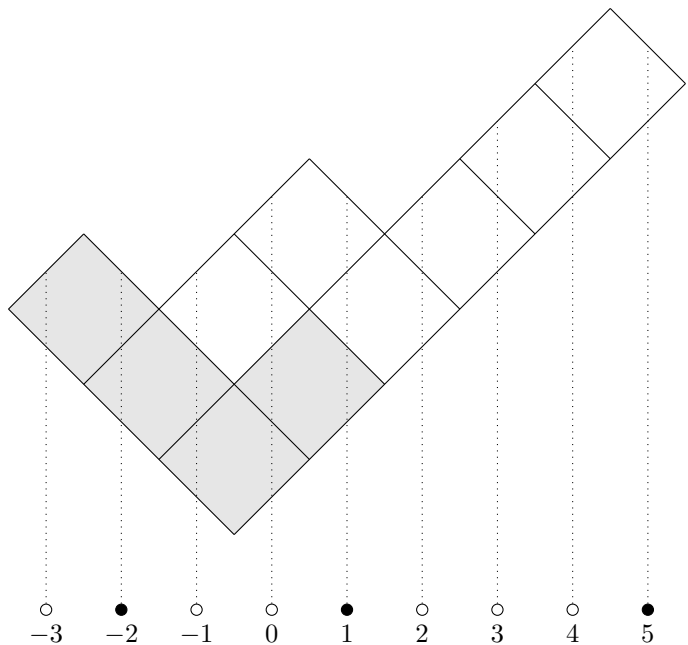
\end{ex}

\subsection{Multi-symmetric functions}

Fix $l\geq 0$ and consider $\Lambda^{(l)}_{t_{1},t_{2}}:=\Lambda_{t_{1},t_{2}}^{\otimes l}$. This vector space is generated as an algebra over $\mathbb{Q}(t_{1},t_{2})$ by the ``colored" power sums $p^{(i)}_{n}$ for $0\leq i \leq l-1$ and $n \geq 1$, which is the element with $p_{n}$ in the $i$th tensorand and $1$ elsewhere.  

We will need a few algebra endomorphisms of $\Lambda^{(l)}_{t_{1},t_{2}}$. Let $\pneg$ be the $\mathbb{Q}(t_1,t_2)$-algebra endomorphism defined by $\pneg(p^{(i)}_{j})=p^{(-i)}_{j}$, where the superscript is understood modulo $l$. Let $\swap$ be the $\mathbb{Q}$-algebra automorphism of $\Lambda^{(l)}_{t_{1},t_{2}}$ given by swapping $t_{1}$ and $t_{2}$. Let $\inv$ be the $\mathbb{Q}$-algebra automorphism given by inverting $t_{1}$ and $t_{2}$.

We will also need the so-called matrix plethysms. For an indeterminant $x$, we define a $\mathbb{Q}(t_{1},t_{2})$-algebra automorphism $\Gamma_{x}$ by 
$$\Gamma_{x}(p_n^{(i)})= p_n^{(i)}-x^n p_n^{(i-1)}$$

The superscripts in the previous two formulas should be taken modulo $l$.

\subsection{Extended Hall pairing}

Given $\lambda^{\bullet}=(\lambda^{0},\ldots,\lambda^{l-1}) \in \parts^{l}$, the corresponding multi-Schur function is
$$
s_{\lambda^{\bullet}}:=s_{\lambda^{(0)}} \otimes \ldots \otimes s_{\lambda^{(l-1)}} \in \Lambda^{(l)}_{t_1,t_2}
$$
The set $\{s_{\lambda^{\bullet}} \, \mid \, \lambda^{\bullet} \in \parts^{l}\}$ is a basis of $\Lambda^{(l)}_{t_1,t_2}$ over $\mathbb{Q}(t_1,t_2)$. There is an extended Hall pairing on $\Lambda^{(l)}_{t_1,t_2}$, defined on multi-Schur functions by
$$
\langle s_{\lambda^{\bullet}}, s_{\mu^{\bullet}} \rangle^{(l)} = \delta_{\lambda^{\bullet}, \mu^{\bullet}}
$$
Combining this fact with the bijection \eqref{quotcore}, we see that the set
$$
\{s_{\quot_{l}(\lambda)} \, \mid \, \lambda \in \parts, \core_{l}(\lambda)=\mu\}
$$
is a basis of $\Lambda^{(l)}_{t_1,t_2}$ for any fixed $l$-core partition $\mu$. Let $\vec{s}_{\lambda}=s_{\quot_{l}(\lambda)}$.

The orthonormality of the multi-Schur functions means that
\begin{equation}\label{schurcauchy}
\sum_{\substack{\lambda \in \parts \\ \core(\lambda)=\mu}} \vec{s}_{\lambda}[x] \vec{s}_{\lambda}[y]= \prod_{i=0}^{l-1} \exp\left(\sum_{n \geq 1} \frac{1}{n} p^{(i)}_{n}[x] p^{(i)}_{n}[y]\right)
\end{equation}
where $\mu$ is a fixed $l$-core partition. In the equality \eqref{schurcauchy}, only finitely many terms contribute to each term of fixed bidegree. The identity \eqref{schurcauchy} follows from the preceding discussion along with the formula for the reproducing kernel for the Hall pairing when $l=1$.

\subsection{Wreath Macdonald polynomials}
Wreath Macdonald polynomials were originally defined by Haiman in \cite{Haiman1}. Their existence was proved using geometric techniques in \cite{BFwreath}. We review their definition, following the expositions given in \cite{OrrShim} and \cite{Wen}.

Let $\unlhd$ denote the dominance order on partitions. We will write $\lambda \unlhd_{l} \mu$ to mean that $\core_{l}(\lambda)=\core_{l}(\mu)$ and $\lambda \unlhd \mu$. 

\begin{defn}[\cite{Haiman1}]
    The Wreath Macdonald polynomial $H_{\lambda}$ is the unique element of $\Lambda_{t_{1},t_{2}}^{(l)}$ satisfying the following conditions:

    \begin{itemize}
        \item $\Gamma_{t_{1}}(H_{\lambda})\in \bigoplus_{\mu  \unrhd_l \lambda} \Q(t_{1},t_{2}) s_{\mu}$
        \item $\Gamma_{t_{2}^{-1}}(H_{\lambda})\in \bigoplus_{\lambda \unrhd_l \mu} \Q(t_{1},t_{2}) s_{\mu}$
        \item $\langle s_{((|\lambda|),\emptyset,\ldots,\emptyset)}, H_{\lambda}\rangle^{(l)} =1$ 
    \end{itemize}
\end{defn}
Note that in the last condition, $((|\lambda|),\emptyset,\ldots,\emptyset)$ is an $l$-tuple of partitions.

\subsection{Wreath Hall pairing}

We also consider the $(t_{1},t_{2})$-deformation of the extended Hall pairing defined by
\begin{equation}\label{wreathhall}
\langle f, g \rangle_{q,t}^{(l)}:= \left \langle f, \pneg \, \Gamma_{t_1} \, \pneg \, \Gamma_{t_2} \, \pneg  \,g \right \rangle^{(l)}
\end{equation}
For a fixed $l$-core partition $\mu$, the collections $\{H_{\lambda} \, \mid \, \core_{l}(\lambda)=\mu\}$ and $\{\inv \, \swap \, H_{\lambda'} \, \mid \, \core_{l}(\lambda)=\mu\}$ are orthogonal bases under this scalar product and
\begin{align}\nonumber
N_{\lambda}&:= \langle H_{\lambda}, \inv \,  \swap \, H_{\lambda'} \rangle \\
&=  \prod_{\substack{\square \in \lambda \\ h_{\lambda}(\square) \equiv 0 \mod l}} (1-t_1^{l(\square)+1}t_2^{-a(\square)})(1-t_1^{-l(\square)}t_2^{a(\square)+1}) \label{orthog}
\end{align}
see \cite{OrrShim} section 3.6.

\subsection{Cauchy identities}

\begin{prop}
Fix $l$ and an $l$-core partition $\mu$. Then
    \begin{multline*}
\sum_{\substack{\lambda \in \parts \\ \core(\lambda)=\mu}} N_{\lambda}^{-1} H_{\lambda}[x] \inv( \swap (H_{\lambda'}[y])) \\ = \exp\left(\sum_{i=0}^{l-1} \sum_{n \geq 1}  p^{(i)}_{n}[x] \pneg(\Gamma_{t_2}^{-1}( \pneg( \Gamma_{t_1}^{-1}( \pneg(p^{(i)}_{n}[y]))))) \right) \\
=\exp\left(\sum_{i=0}^{l-1} \sum_{n \geq 1}\frac{p^{(i)}_{n}[x]}{n (1-t_1^{nl})(1-t_2^{nl})} \sum_{\substack{j,k=0}}^{l-1}  t_1^{nj} t_2^{nk} p^{(-i-j+k)}_{n}[y]\right)
     \end{multline*}
\end{prop}
\begin{proof}
    The first line follows \eqref{schurcauchy}, \eqref{wreathhall}, and \eqref{orthog}. The second line follow from the formula 
  \begin{equation*}
\Gamma^{-1}_{x}(p^{(i)}_{n})=(1-x^{nl})^{-1}\sum_{j=0}^{l-1}  x^{nj} p^{(i-j)}_{n}
\end{equation*}

\end{proof}

The following evaluation formula was recently proved by Wen and Romero.

\begin{prop}[\cite{josh}]\label{prop: classicaleval}
Let $0 \leq m \leq l-1$. If $\core(\lambda)=\emptyset$, then
    $$
    H_{\lambda}|_{p^{(j)}_{k} =\delta_{0,j} - \delta_{m,j} (-u)^{k}}  = \prod_{\substack{\square=(a,b) \in \lambda \\ a-b \equiv m \mod l}}(1+u t_1^{b-1} t_2^{a-1})
    $$
\end{prop}

Combining the above proposition with the Cauchy identity gives the following.
\begin{lem}\label{classicalformula}
Let $0 \leq m \leq l-1$. Then
      \begin{multline*}
\sum_{\substack{\lambda \in \parts \\ \core(\lambda)=\emptyset}} N_{\lambda}^{-1}  H_{\lambda} \prod_{\substack{(a,b) \in \lambda' \\ a-b \equiv m \mod l}} (1+u t_1^{1-a} t_2^{1-b}) \\
= \exp\left(\sum_{i=0}^{l-1} \sum_{n \geq 1}\frac{p^{(i)}_{n}}{n (1-t_1^{nl})(1-t_2^{nl})} \left(C^{(i)}_{0,n}-C^{(i)}_{m,n}(-u)^{n}\right)\right)
    \end{multline*}
    where 
    $$
C^{(i)}_{m,n}=\sum_{\substack{j,k=0 \\ -i-j+k \equiv m \mod l}}^{l-1}  t_1^{nj} t_2^{nk}
    $$
\end{lem}

The left hand side of the previous lemma can also be written as
$$
\sum_{\substack{\lambda \in \parts \\ \core(\lambda)=\emptyset}} N_{\lambda}^{-1}  H_{\lambda} \prod_{\substack{(a,b) \in \lambda \\ a-b \equiv -m \mod l}} (1+u t_1^{1-b} t_2^{1-a})
$$

\section{Nakajima Quiver Varieties}\label{NQVsection}

In this section, we review the geometry relevant for our paper: affine type $A$ Nakajima quiver varieties. 

\subsection{Cyclic quiver varieties}

Let $Q$ be the $\hat{A}_{l-1}$ quiver, which has vertex set $I=\{0,1,\ldots,l-1\}$ and edges $i\to i+1$ for $0 \leq i \leq l-1$. Here and elsewhere, we use the convention that indices corresponding to vertices of $Q$ are taken modulo $l$. 

Choose $\dv,\dw \in \mathbb{Z}^{I}_{\geq 0}$. Let $V_{i}$ and $W_{i}$ be $\mathbb{C}$-vector spaces such that $\dim_{\mathbb{C}} V_{i}=\dv_{i}$ and $\dim_{\mathbb{C}} W_{i}=\dw_{i}$.

Let
$$
\text{Rep}_{Q}(\dv,\dw)=\bigoplus_{i\in I}\Hom(V_{i},V_{i+1}) \oplus \bigoplus_{i \in I} \Hom(W_{i},V_{i})
$$
The group $G_{\dv}=\prod_{i \in I} GL(V_{i})$ acts naturally on $\text{Rep}_{Q}(\dv,\dw)$, inducing a Hamiltonian action on $T^*\text{Rep}_{Q}(\dv,\dw)$ and a moment map
$$
\mu: T^*\text{Rep}_{Q}(\dv,\dw) \to \text{Lie}(G_{\dv})
$$
Choose a character $\theta:G_{\dv} \to \mathbb{C}^{\times}$.

\begin{defn}[\cite{NakALE,NakQv}]
    The Nakajima quiver variety associated to the data $Q$, $\dv$, $\dw$, and $\theta$ is the geometric invariant theory quotient
    $$
\qv_{Q,\theta}(\dv,\dw):=\mu^{-1}(0)/\!\!/_{\theta} G_{\dv}
$$
\end{defn}

A point in $\qv_{Q,\theta}(\dv,\dw)$ is represented by a collection of linear maps $\{A_{i},B_{i},a_{i},b_{i}\}_{i \in I}$ where $A_{i} \in \Hom(V_{i},V_{i+1})$, $B_{i} \in \Hom(V_{i+1},V_{i})$, $a_{i} \in \Hom(W_{i},V_{i})$, and $b_{i} \in \Hom(V_{i},W_{i})$. We will abuse notation and denote points by $(A,B,a,b)$.

For the remainder of this paper, we take $\theta$ to be the positive stability, which is the character defined by $\theta((g_{i})_{i \in I})=\prod_{i \in I} \det(g_{i})$. We will frequently simplify the notation and write $\qv(\dv,\dw)$ (or sometimes just $\qv$) with the choices of $Q$ and $\theta$ (and $\dv$ and $\dw$) understood. We will also denote by $\delta_{j}$ the framing vector with $1$ in position $j$ and zeros elsewhere.

\subsection{Hilbert schemes of points}
Assume that $\dw=\delta_{0}=(1,0,\ldots,0)$. Then $\qv(\dv,\dw)$ has another description as a moduli space of certain ideals in $\mathbb{C}[x,y]$ which we review next. 

The Hilbert scheme of $m$ points in $\mathbb{C}^{2}$, written $\text{Hilb}^{m}(\mathbb{C}^{2})$, is the moduli space of ideals $J \subset \mathbb{C}[x,y]$ such that $\dim_{\mathbb{C}} (\mathbb{C}[x,y]/J)=m$. It can be constructed as a Nakajima variety for the Jordan quiver with $\dv=(m)$ and $\dw=(1)$.

Let $\Gamma_{l}\subset \mathbb{C}^{\times}$ be the cyclic group of $l$th roots of unity with generator $\zeta=e^{2\pi i/l}$. There is an action of $\Gamma_{l}$ on $\mathbb{C}[x,y]$ defined by
\begin{align*}
\zeta \cdot x &= \zeta x \\
\zeta \cdot y &= \zeta^{-1} y
\end{align*}
This induces an action of $\Gamma_{l}$ on $\text{Hilb}^{m}(\mathbb{C}^{2})$

Then $\qv(\dv,\dw)$ is the $\Gamma_{l}$-fixed component of $\text{Hilb}^{|\dv|}(\mathbb{C}^{2})$ consisting of ideals $J$ such that
$$
\text{char}_{\Gamma_{l}} (\mathbb{C}[x,y]/J)= \sum_{i \in I} \dv_{i} \chi^{i}
$$
where $\chi: \Gamma_{l} \to \mathbb{C}^{\times}$ is the character defined by $\chi(\zeta)=\zeta$.

When $\dv=(n,n,\ldots,n)$ and $\dw=(1,0,\ldots,0)$, we denote 
$$
\qv(n):=\qv(\dv,\dw)
$$

\subsection{Torus action and fixed points}\label{sec: fixedpoints}
Let $\bA=(\mathbb{C}^{\times})^{|\dw|}$. There is an action of the torus $\bT=\bA \times (\mathbb{C}^{\times})^{2}$ on $\qv(\dv,\dw)$. The action of the $\bA$ factor on $\qv$ is induced by the action of diagonal matrices on $W_{i}$. The second component of $\bT$ acts by 
$$
(t_{1},t_{2})\cdot (A,B,a,b) = (t_2 A, t_1 B, a, t_1 t_2 b), \quad (t_{1},t_{2}) \in (\mathbb{C}^{\times})^{2}
$$

Next we describe the torus fixed points of $\qv(\dv,\dw)$, which relies on the ``tensor product" of quiver varieties, see \cite{Naktens}. Choose a decomposition $\dw=\dw'+\dw''$ corresponding to a choice of splitting $W_{i}=W_{i}'\oplus W_{i}''$ and let $\bA' \subset \bA$ act trivially on $W_{i}''$ and with weight $1$ on $W_{i}'$. Letting $a$ be the coordinate on $\bA'$, we abbreviate this setup by writing $\dw=a \dw'+\dw''$. Then
$$
\qv(\dv,\dw)^{\bA'}= \bigsqcup_{\substack{\dv',\dv'' \in \mathbb{Z}^{I}_{\geq 0} \\ \dv'+\dv''=\dv}}\qv(\dv',\dw')\times \qv(\dv'',\dw'')
$$
Iterating this decomposition, one sees that $\bT$-fixed points on $\qv(\dv,\dw)$ are given by products of $(\mathbb{C}^{\times})^{2}$-fixed points on $\qv(\dv',\dw')$ where $|\dw'|=1$.

Due to the rotational symmetry of the quiver, it is sufficient to describe $\qv(\dv,\delta_{0})^{(\mathbb{C}^{\times})^{2}}$ where $\delta_{0}=(1,0,\ldots,0)$. It is known that $(\mathbb{C}^{\times})^{2}$-fixed points on $\text{Hilb}^{m}(\mathbb{C}^{2})$ are in natural bijection with partitions $\lambda$ of $m$. Since the embedding $\qv(\dv,\delta_{0}) \subset \text{Hilb}^{|\dv|}(\mathbb{C}^{2})$ is $(\mathbb{C}^{\times})^{2}$ equivariant, the $(\mathbb{C}^{\times})^{2}$-fixed points of $\qv(\dv,\delta_{0})$ are given by partitions $\lambda$ of $|\dv|$ such that the number of boxes of content $i$ modulo $l$ is equal to $\dv_{i}$ for all $i$, see also \cite{combqv}.

For general $\dw$, $\qv(\dv,\dw)^{\bT}$ is thus in bijection with certain $|\dw|$-tuples of partitions. We denote such a tuple by $(\lambda^{(i,j)})_{\substack{0 \leq i \leq l-1 \\ 1 \leq j \leq \dw_{i}}}$. Whenever we discuss tuples of partitions of this form labeling a torus fixed point, we will use the convention the content of a box $\square \in \lambda^{(i,j)}$ is equal to $c(\square)+i$, where $c$ is the usual content function from section \ref{partitions}.

In summary, we have the following.
\begin{prop}\label{fixedpoints}

$\qv(\dv,\dw)^{\bT}$ is in bijection with $|\dw|$-tuples of partitions $(\lambda^{(i,j)})_{\substack{0 \leq i \leq l-1 \\ 1 \leq j \leq \dw_{i}}}$ such that the total number of boxes with content $i$ modulo $l$ is equal to $\dv_{i}$.
\end{prop}

As a special case, we have the following.

\begin{prop}
    $\qv(n)^{\bT}$ is in natural bijection with partitions of $nl$ with empty $l$-core.
\end{prop}

\subsection{Tautological bundles}

The vector spaces $V_{i}$ appearing in the definition of $\qv(\dv,\dw)$ descend to tautological vector bundles $\tb_{i}$ on $\qv(\dv,\dw)$. If $\dw=\delta_{0}$ and $\lambda \in \qv(\dv,\dw)^{\bT}$, then
$$
\text{char}_{\bT}(\tb_{m}|_{\lambda})=\sum_{\substack{(a,b) \in \lambda \\ a-b \equiv m \mod l}} t_1^{1-b} t_2^{1-a}
$$

Similarly, the spaces $W_{i}$ descend to topologically trivial bundles $\mathcal{W}_{i}$.

The following tautological classes are our main interest.

\begin{defn}\label{descdef}
Fix $\dv, \dw$, and let $\tb_{m}$ be the $m$th tautological bundle on $\qv(\dv,\dw)$. Let $\tau_{m}=\sum_{j=0}^{\dv_{m}} \extpow^{j} \tb_{m}$.
\end{defn}

\subsection{Connection to symmetric functions}

\begin{thm}\label{Ksym}(\cite{BFwreath,losev})
    There is an isomorphism of graded vector spaces over $\mathbb{Q}(t_1,t_2)$
    $$
\bigoplus_{\dv} K_{\bT}\left( \qv(\dv,\delta_{0}) \right)_{loc}  \cong \Lambda ^{(l)}_{t_{1},t_{2}} \otimes_{\mathbb{Q}} \mathbb{Q}[Q]
    $$
    such that $[I_{\lambda}] \mapsto H_{\lambda}\otimes e^{\core_{l}(\lambda)}$ where $[I_{\lambda}]$ is the $K$-theory class of the skyscraper sheaf at the fixed point $\lambda \in \bigsqcup_{\dv} \qv(\dv,\delta_{0})^{\bT}$. Here $\mathbb{Q}[Q]$ is the group algebra of $Q$, where $Q$ is the finite type $A_{l-1}$ root lattice and $\core_{l}(\lambda)$ is identified with an element of $Q$ by Proposition \ref{coreroot}.
\end{thm}

Let $\qv:=\bigsqcup_{n \geq 0} \qv(n)$. Then, under the identification of Theorem \ref{Ksym}, the left hand side of Lemma \ref{classicalformula} should be interpreted geometrically as
$$
\sum_{\substack{\lambda \in \parts \\ \core(\lambda)=\emptyset}} \frac{[I_{\lambda}]}{\bigwedge^{\bullet}\left( T \qv^{\vee}|_{\lambda} \right)} \left(\cver{\tau_{-m}}_{\lambda}\right)\bigg|_{\substack{z_{i}=0 , \,  0 \leq i \leq l-1}}
$$
where $\cver{\tau_{-m}}_{\lambda}$ is the capped vertex function with descendant $\tau_{-m}$ to be defined in the next section. Because of the equality
$$
f = \sum_{\substack{\lambda \in \parts  \\ \core(\lambda)=\emptyset}} \frac{f|_{\lambda} [I_{\lambda}]}{\bigwedge^{\bullet}\left(T\qv^{\vee}|_{\lambda}\right)}
$$
in localized equivariant $K$-theory, the left hand side of Lemma \ref{classicalformula} is simply $\cver{\tau_{-m}}|_{z_{i}=0, \, 0 \leq i \leq l-1}$.

\section{Quasimap counts}

In this section, we will review the main objects appearing in the $K$-theoretic enumerative geometry of Nakajima varieties. Our main references for this section are \cite{qm,Ok17,OkSm}. Since we will not contributed anything new here, we will pass quickly over the details.

\subsection{Moduli of quasimaps}

As in previous sections, we fix $l \geq 0$ and consider quiver varieties $\qv:=\qv(\dv,\dw)$ for the type $\hat{A}_{l-1}$ quiver. For $d \in \mathbb{Z}^{l}$, let $\qm^d:=\qm^{d}(\mathbb{P}^{1}\to \qv)$ denote the moduli space of stable degree $d$ quasimaps from $\mathbb{P}^1$ to $\qv$, see  \cite{Ok17} section 4. In particular, a quasimap from $\mathbb{P}^{1}$ to $\qv$ is a map from $\mathbb{P}^{1}$ to the quotient stack $\stackqv:= [\mu^{-1}(0)/G_{\dv}]$. 

Let $\mathbb{C}_q^{\times}$ be the standard torus acting on $\mathbb{P}^1$ by scaling coordinates with weight $q$:

$$[z_1:z_2]\mapsto [qz_1:z_2]$$

This induces a $\mathbb{C}^{\times}_{q}$ action on $\qm^{d}$. Similarly, the $\bT$-action on $\qv$ induces an action on $\qm^{d}$. Let $G=\bT \times \mathbb{C}^{\times}_{q}$.

There is a subset of the moduli space that we will primarily be interested in: the open subset of quasimaps of degree $d$ nonsingular at $0$. By definition, this is $\qm^{d}_{\ns 0}= \{f \in \qm^{d} \, \mid \, f(0) \in \qv \subset \stackqv\}$. By definition, there are $G$-equivariant evaluation maps
$$
\begin{tikzcd}
    & \qm^{d}_{\ns 0}  \arrow{ld}{\ev_{0}} \arrow{rd}{\ev_{\infty}} & \\
    \qv & & \stackqv
\end{tikzcd}
$$
By \cite{Ok17} sections 6 and 7, there is a space $\qm^{d}_{\rel 0}$ and a commutative diagram
\begin{equation}\label{qmdiagram}
\begin{tikzcd}
(\qm^{d}_{\ns 0})^{\mathbb{C}^{\times}_{q}}  \arrow[hookrightarrow]{r} \arrow{rd} &   \qm^{d}_{\ns 0} \arrow[hookrightarrow]{r} \arrow{d}{\ev_{0}} & \qm^{d}_{\rel 0}  \arrow{ld}[below]{\hat{\ev}_{0}}\\
& \qv &
\end{tikzcd}
\end{equation}
such that the two diagonal maps are proper. We will abuse notation and write $\ev_{0}$ for the southeast pointing map.

It is known that $K_{G}(\stackqv)=K_{G \times G_{\dv}}(pt)$, which is generated over $K_{G}(pt)$ by symmetric Laurent polynomials in the Chern roots of $\tb_{i}$ for $0 \leq i \leq l-1$. We denote elements of $K_{G \times G_{\dv}}(pt)$ by $\tau$. Abusing notation, we denote their restriction to $K_{G}(\qv)$ by the same symbol.

Moduli spaces of quasimaps come equipped with a perfect obstruction theory, and hence natural virtual classes. We denote by $\vrs^{d}$ the symmetrized virtual structure sheaf on a moduli space of quasimaps of degree $d$, see \cite{Ok17} section 6. The symmetrization requires a choice of a polarization, which we take to be the virtual bundle
$$
T^{1/2}=\hbar \sum_{i=0}^{l-1} \Hom(\tb_{i+1},\tb_{i}) + \hbar \sum_{i=0}^{l-1} \Hom(\tb_{i},\mathcal{W}_{i})- \hbar \sum_{i=0}^{l-1} \Hom(\tb_{i},\tb_{i})
$$

\subsection{Vertex}

Okounkov makes the following definition.

\begin{defn}(Section 7 of \cite{Ok17})
    The bare vertex with descendant $\tau \in K_{G \times G_{\dv}}(pt)$ is the generating function

    $$\ver{\tau}(z):  = \mathcal{K}^{-1/2} \sum_{d} \ev_{0,*}\left(\qm^d_{\ns 0},\vrs^{d} \otimes \ev_{\infty}^{*}(\tau) \right)z^d \in K_{G}\left(\qv\right)_{loc}[[z]]$$
    where $\mathcal{K}$ is the canonical bundle of $\qv$.

\end{defn}
In the previous definition, the pushforward $\ev_{0,*}$ is defined by $\mathbb{C}^{\times}_{q}$-localization using diagram \eqref{qmdiagram}. In the above definition we use $z^d=z_0^{d_0}\cdots z_{l-1}^{d_{l-1}}$ to keep track of degrees of quasimaps in what are usually called $\textit{K\"{a}hler parameters}$. The collection of $d \in \mathbb{Z}^{l}$ such that $\qm^{d} \neq \emptyset$ forms a cone; the notation $[[z]]$ should be understood as representing a completion of the semigroup algebra of this cone.

For quiver varieties with finitely many torus fixed points and stability condition $\theta=\pm (1,1,\ldots,1)$ (which includes all varieties studied in this paper), bare vertex function can be computed by equivariant localization. For examples of computations in this regard, see \cite{D5Vertex,dinksmir2,liu,dinksmir}.

\subsection{Capped vertex}

Using the space $\qm^{d}_{\rel 0}$ in \eqref{qmdiagram}, Okounkov defines the following.

\begin{defn}(Section 7 of \cite{Ok17})
    The capped vertex with descendant $\tau\in K_{G \times G_{\dv}}(pt)$ is the generating function
    $$\cver{\tau}(z):=\canon^{-1/2}\sum_{d} \hat{\ev}_{0,*}\left(\qm^d_{\rel 0},\vrs^{d} \otimes \ev^{*}_{\infty}(\tau)\right)z^d \in K_{G}\left(\qv\right)_{}[[z]]$$
\end{defn}
Since $\hat{\ev}_{0}$ is proper, the capped vertex lives in non-localized equivariant $K$-theory, in contrast the bare vertex.

A remarkable rigidity theorem proved by Okounkov in \cite{Ok17} is that, for sufficiently large framing, the capped vertex function obeys a property known as \textit{large framing vanishing}.

\begin{thm}(Theorem 7.5.23 in \cite{Ok17}) \label{LargeFrame}
    Let $\tau\in K_{G \times G_{\dv}}(pt)$. There exists some framing vector $\dw(\tau)$ such that for any $\dw\geq \dw(\tau)$\footnote{Here we mean $\dw \geq \dw' \iff \dw_{i} \geq \dw'_{i}$ for all $i$.} the capped vertex with descendant $\tau$ for $\qv(\dv,\dw)$ is purely classical:
    $$\cver{\tau}(z)= \tau$$
\end{thm}

\subsection{Capping operator}
There are also quasimap spaces obtained by imposing a nonsingular condition at $\infty$ and a relative condition at $0$, denoted by $\qm_{\substack{\rel 0,\\ \ns \infty}}^{d}$.

\begin{defn}(Section 8 of \cite{Ok17})\label{cap}
    The capping operator is the generating function
    $$\Psi(z):= \sum_d \ev_{0,*}\otimes \ev_{\infty,*}\left(\qm^d_{\substack{\rel 0,\\ \ns \infty}},\vrs^{d}\right) z^d \in K_{G}\left(\qv\right)^{\otimes 2}_{loc}[[z]]$$
\end{defn}

Using the natural bilinear pairing on equivariant $K$-theory, $(\alpha,\beta)\mapsto \chi(\qv,\alpha \otimes \beta)$, one can view $\Psi(z)$ as an element of $\text{End}\left(K_{G}(\qv)_{loc} \right)[[z]]$. Computing the capped vertex by localization gives the following.

\begin{thm}[Section 7.4 of \cite{Ok17}]\label{VPsiV} The following equation holds:
$$\cver{\tau}(z) = \Psi(z) \ver{\tau}(z)$$
\end{thm}

In \cite{OkSm}, Okounkov and Smirnov showed that the operator $\Psi(z)$ is the fundamental solution of the \textit{quantum difference equation}. This equation is described by operators in a certain Hopf algebra acting on the equivariant $K$-theory of $\qv$.

\section{Large framing vanishing}

In this section, we prove some technical results about the vanishing of quantum corrections for $\cver{\tau}$ for cyclic quiver varieties. Our main goal is to prove the following result.

\begin{thm}\label{lfv}
    Let $\dv=(n,n,\ldots,n)$ and $\dw=\delta_{0}+\delta_{m}$. Let $x=(\lambda_{0},\lambda_{m})\in \qv(\dv,\dw)^{\bT}$ such that $\lambda_{0}$ (hence $\lambda_{m}$) has empty core. Then $\cver{\tau_{m}}_{x}=\tau_{m}|_{x}$. 
\end{thm}

We will prove this theorem at the end of section \ref{sec: last lfv section} after setting up the necessary prerequisites in the next three sections.

\subsection{Torus fixed quasimaps and localization}

In \cite{Ok17}, Okounkov proves that for an arbitrary descendant, the capped vertex will be purely classical for sufficiently large framing. By the logic of the proof in \cite{Ok17} section 7.5, $\cver{\tau}_{p}=\tau|_{p}$ whenever $\lim_{q \to \infty} \ver{\tau}_{p}=\tau|_{p}$. 

Let $f \in \left(\qm_{\ns 0}\right)^{\mathbb{C}^{\times}_{q}}$ be a nonconstant quasimap. The quasimap $f$ provides the data of vector bundles $\mathscr{V}_{i}$ over $\mathbb{P}^{1}$ which are linearized such that $\mathbb{C}^{\times}_{q}$ acts trivially on $\mathscr{V}_{i}|_{0}$. We also obtain a virtual bundle $\qmpol$ over $\mathbb{P}^{1}$. Then $\ev^{*}_{\infty} (\tb_{i})|_{f}=\mathscr{V}_{i}|_{\infty}$. More generally, $\ev^{*}_{\infty}(\tau)|_{f}$ is obtained by replacing the bundles $\tb_{i}$ appearing in $\tau$ by $\mathscr{V}_{i}$ and taking the fiber over $\infty$. So we denote $\tau(\mathscr{V})|_{\infty}:=\ev^{*}_{\infty}(\tau)|_{f}$. The vector space $\tau(\mathscr{V})|_{\infty}$ carries a nontrivial action of $\mathbb{C}^{\times}_{q}$. As bundles over $\mathbb{P}^{1}$, we have
$\mathscr{V}_{i}=\bigoplus_{j} \mathcal{O}(d_{i,j})$. Due to our choice of stability condition, $d_{i,j} \geq 0$ for all $i,j$.

By (7.5.14) of \cite{Ok17}, the $q\to \infty$ limit of the contribution of $f$ to the localization formula for $\ver{\tau}$ behaves like
\begin{equation}\label{asymp}
\tau(\mathscr{V})|_{\infty} q^{-\deg \qmpol_{p}}
\end{equation}
where $\qmpol_{p}$ is the virtual sub-bundle of $\qmpol$ consisting of terms with positive degree.

For $\tau=\tau_{i,j}:=\extpow^{j}\tb_{i}$ it is easy to see that 
$$
\deg_{q} \tau_{i,j}(\mathscr{V})|_{\infty}\leq \deg \mathscr{V}_{i}
$$
So the $q \to \infty$ limit of the \eqref{asymp} will be $0$ if
$$
\deg \mathscr{V}_{i} -\deg \qmpol_{p} <0
$$
By localization, it suffice to prove this for $f \in (\qm_{\ns 0})^{\bT \times \mathbb{C}^{\times}_{q}}$, i.e. for quasimaps fixed additionally by $\bT$.

As discussed in \cite{Ok17} section 7.5, the dependence of the vertex on the polarization is through a simple shift. For the remainder of this section, we take the polarization to be
$$
\sum_{i=0}^{l-1} \Hom(\tb_{i},\tb_{i+1}) + \sum_{i=0}^{l-1} \Hom(\mathcal{W}_{i},\tb_{i})-\Hom(\tb_{i},\tb_{i})
$$

\subsection{Torus fixed quasimaps and reverse plane partitions}

Let $f \in \qm_{\ns 0}\left(\mathbb{P}^{1} \to \qv(\dv,\dw)\right)^{\bT \times \mathbb{C}^{*}_{q}}$ be a quasimap such that $f(0)=x$. Let us recall the combinatorial description of such $f$ in terms of reverse plane partitions.

The moduli space of quasimaps to $\qv(\dv,\dw)$ which are fixed by $\bA$ is isomorphic to the moduli space of quasimaps to $\qv(\dv,\dw)^{\bA}$, see \cite{Ok17} section 7.3. By Proposition \ref{fixedpoints}, it suffices to describe torus fixed quasimaps when $|\dw|=1$. In this case, the fixed point $x$ is given by a single partition $\lambda$. 

As recalled above, a quasimap from $\mathbb{P}^{1}$ to $\qv(\dv,\dw)$ provides the data of vector bundles $\mathscr{V}_{i}$. These split into a direct sum of line bundles $\mathscr{V}_{i}=\bigoplus_{j} \mathcal{O}(d_{i,j})$ where $0 \leq i \leq l-1$ and $1 \leq j \leq \dv_{i}$. If the quasimap is $\bT\times\mathbb{C}^{\times}_{q}$-fixed, then this splitting is additionally graded by the $\bT$-weights of $\tb_{i}|_{\lambda}$ which thus puts the numbers $d_{i,j}$ into canonical bijection with boxes $\square \in \lambda$. One can additionally show, see for example \cite{dinkinsthesis} section 5.1.3, that the integers $d_{i,j}$ must form a ``reverse plane partition" over $\lambda$. In our notation, this means that we have a collection of non-negative integers $d_{\square}$ for each $\square \in \lambda$ such that $d_{\square} \leq d_{\square'}$ whenever $\square'$ is immediately above or right of $\square$, with directions as in section \ref{partitions}. Furthermore, the reverse plane partition uniquely determines the $\bT\times \mathbb{C}^{\times}_{q}$-fixed quasimap.

\subsection{Removing $l$-strips}\label{sec: last lfv section}

Let $\dv$ and $\dw$ be arbitrary. Let $x\in \qv(\dv,\dw)^{\bT}$. Recall from Proposition \ref{fixedpoints} that $x$ is equivalent to a certain tuple of partitions. Assume that not all of these partitions are $l$-core partitions. So there is some partition $\lambda$ which has a removable $l$-strip $\alpha$. Let $\lambda'$ be a partition obtained by removing $\alpha$ from $\lambda$, i.e., $\lambda'=\lambda \setminus \alpha$. The tuple of partitions for $x$, but with $\lambda$ replaced by $\lambda'$, indexes a $\bT$-fixed point $x'$ on some quiver variety $\qv(\dv',\dw)^{\bT}$. 

Let $f \in \qm_{\ns 0}\left(\mathbb{P}^{1} \to \qv(\dv,\dw)\right)^{\bT \times \mathbb{C}^{*}_{q}}$ be a quasimap such that $f(0)=x$. As we recalled above, $f$ is equivalent to a tuple of reverse plane partitions over the partitions for $x$. By disregarding the boxes over $\alpha$ in this tuple of reverse plane partitions, we obtain $f' \in \qm_{\ns 0}\left(\mathbb{P}^{1} \to \qv(\dv',\dw)\right)^{\bT \times \mathbb{C}^{\times}_{q}}$ such that $f'(0)=x'$. 

The quasimaps $f$ and $f'$ provide vector bundles $\mathscr{V}_{i,f}$ and $\mathscr{V}_{i,f'}$. By construction, $\mathscr{V}_{i,f}=\mathscr{V}_{i,f'} \oplus \mathcal{O}(d_{i})$ where $d_{i}\geq 0$ is the component of the reverse plane partition for $f$ corresponding to the unique box in $\alpha$ of content $i$. We remind the reader of our convention for contents of boxes in tuples of partitions given in section \ref{sec: fixedpoints}. The following is clear.

\begin{lem}\label{compareV}
    $$
\deg \mathscr{V}_{i,f}= \deg \mathscr{V}_{i,f'} + d_{i}
    $$
  \end{lem}

To distinguish the virtual bundles $\qmpol$ provided by $f$ and $f'$, we denote them by $\qmpol_{f}$ and $\qmpol_{f'}$. As above, we also have $\qmpol_{p,f}$ consisting of terms with positive degree and similarly for $f'$.

\begin{lem}
    $$
    \deg \qmpol_{p,f'}+ \sum_{i} (\dv_{i-1}-\dv_{i}+\dw_{i})d_{i} \leq \deg \qmpol_{p,f}
    $$
\end{lem}
\begin{proof}
The only possible negative contributions to $\deg \qmpol_{p,f}$ which appear in $\qmpol_{p,f}$ but not in $\qmpol_{p,f'}$ are
$$
-\sum_{i} \deg \text{Hom}\left(\mathscr{V}_{i}',\mathcal{O}(d_{i}) \right)=-\sum_{i} \left((\dv_{i}-1) d_{i} -\deg \mathscr{V}_{i}'\right)
$$
Some of the additional positive terms in $\deg \qmpol_{p,f}$ are 
\begin{multline*}
\sum_{i} \deg \text{Hom}(\mathscr{V}_{i-1}',\mathcal{O}(d_{i})) + \deg \text{Hom}(\mathcal{O}(0)^{\oplus \dw_{i}},\mathcal{O}(d_{i})) \\
=\sum_{i} \left((\dv_{i-1}-1) d_{i} - \deg \mathscr{V}_{i-1}'+ \dw_{i} d_{i}\right)
\end{multline*}
Adding these together and recalling that the quiver is cyclic gives the result.
\end{proof}

\begin{cor}
    If $\dv_{i-1}-\dv_{i}+\dw_{i} \geq 0$ for all $i$, then 
    $$
\deg \qmpol_{p,f'} \leq \deg \qmpol_{p,f}
    $$
\end{cor}

\begin{cor}\label{comparepol}
    If $\dv_{i}=n$ for all $i$, then $\deg \qmpol_{p,f'} + \sum_{i} \dw_{i} d_{i} \leq \deg \qmpol_{p,f}$. 
\end{cor}

\begin{proof}[Proof of Theorem \ref{lfv}]

Now we specialize to the case of $\qv=\qv(\dv, \delta_{0}+\delta_{m})$ where $\dv=(n,n,\ldots,n)$. Let $x \in \qv^{\bT}$ be a torus fixed point given by pair of partitions $(\lambda_0,\lambda_m)$ which both have empty core. We will prove that $\cver{\tau_{m,j}}_{x}$ is purely classical by induction on the total number $l$-strips in $\lambda_{0}$ and $\lambda_{m}$. 

As the base case, we consider the situation where there is a single $l$-strip among the two partitions. For concreteness, we assume $\lambda_{0}=\emptyset$ and $\lambda_{m}$ is a single $l$-strip, which in this case must be a hook with $l$ boxes. A similar argument handles the case where $\lambda_{0}$ is an $l$-strip and $\lambda_{m}=\emptyset$. Let $f \in \left(\qm_{\ns 0}\right)^{\bT \times \mathbb{C}^{\times}_{q}}$ be a nonconstant quasimap such that $f(0)=x$. Write $ \mathscr{V}_{i}=\mathcal{O}(d_{i})$. Let $d'$ and $d''$ the degrees corresponding to the two outermost boxes in the hook. Since $f$ is nonconstant, at least one of $d'$ or $d''$ must be strictly positive. Also, 
\[
\deg \qmpol_{p,f}=d_{0}+d_{m} + \sum_{\substack{i \\ d_{i+1}>d_{i}}} (d_{i+1}-d_{i})=d_{0}+d_{m} + d''-d_{m}+\max(0,d'-d'')
\]
where the first equality is by definition and the second is because $\lambda_{m}$ is a hook. Since $\deg \mathscr{V}_{m,f}=d_{m}$, we have
\begin{multline*}
\deg \mathscr{V}_{m,f} - \deg \qmpol_{p,f} \leq d_{m}-d_{0}-d''-\max(0,d'-d'') \\
\leq-d'' - \max(0,d'-d'')<0
\end{multline*}
Thus the $q \to \infty$ limit of \eqref{asymp} vanishes and $\cver{\tau_{m,j}}_{x}$ is purely classical. This handles the base case.

For the inductive step, we let $\lambda_{0}$ and $\lambda_{m}$ be arbitrary partitions, at least one of which contains a removable $l$-strip. Let $f$ be a nonconstant $\bT\times \mathbb{C}^{\times}_{q}$-fixed quasimap such that $f(0)=x$. As in the first paragraph of this subsection, we can remove an $l$-strip to obtain a fixed point $x'$ on another quiver variety, along with a quasimap $f'$ with $f'(0)=x'$. 

If $f'$ is nonconstant, then by Lemma \ref{compareV}, Corollary \ref{comparepol}, and the inductive hypothesis, we see that
\begin{align*}
&\deg \mathscr{V}_{m,f} - \deg \qmpol_{p,f} = \deg \mathscr{V}_{m,f'}+d_{m}-\deg \qmpol_{p,f}   \\
&\leq \deg \mathscr{V}_{m,f'}+d_{m}-\deg \qmpol_{p,f'} -d_{0}-d_{m} \leq \deg \mathscr{V}_{m,f'} - \deg \qmpol_{p,f'}<0
\end{align*}

Finally, suppose that $x'$ and $f'$ cannot be chosen in such a way that $f'$ is nonconstant. This implies that only one of $\lambda_{0}$ or $\lambda_{m}$ contains a removable $l$-strip and that the other is empty. Without loss of generality, we can assume that $\lambda_{0}=\emptyset$. Furthermore, all the boxes in the reverse plane partition over $\lambda_{m}$ corresponding to $f$ lie over a single $l$-strip $\alpha$ of $\lambda_{m}$. For two boxes $a,b \in \alpha$, we write $a \to b$ if $a$ is an inner corner, $b$ is an outer corner, $b$ lies in the hook based at $a$, and $c(b)>c(a)$. Let $d_{l}$ and $d_{r}$ be the degrees corresponding to the leftmost and rightmost boxes of $\alpha$, respectively. Let $d_{a}$ be the degree for a box $a \in \alpha$, and let $d_{0}$ and $d_{m}$ be the degrees for the (unique) boxes in $\alpha$ of content $0$ and $m$. Then
$$
\deg \qmpol_{p,f}= d_{0}+d_{m}+\sum_{a \to b} (d_{b}-d_{a}) + \max(0,d_{l}-d_{r})
$$
So
\begin{equation}\label{degineq}
\deg \mathscr{V}_{m,f} - \deg \qmpol_{p,f}=-d_{0} -\sum_{a \to b}(d_{b}-d_{a}) -\max(0,d_{l}-d_{r})
\end{equation}
If $d_{b}-d_{a}>0$ for any $a \to b$, then \eqref{degineq} is negative and we are done. Assume that $d_{b}=d_{a}$ for all $a \to b$. Then if $d_{b_{0}}>0$ for some $b_{0}$ (which must be the case for some $b_{0}\in \alpha$ since $f$ is nonconstant), then $d_{b}>0$ for all boxes $b \in \alpha$ such that $c(b)<c(b_{0})$. Thus if $d_{r}>0$, then $d_{0}>0$ and \eqref{degineq} is negative. If $d_{r}=0$, then $\max(0,d_{l}-d_{r})>0$ and \eqref{degineq} is negative. This exhausts all possibilities.

Hence the $q\to \infty$ limit of \eqref{asymp} vanishes and $\cver{\tau_{m,j}}_{x}$ is purely classical.

\end{proof}

\begin{rem}
    We believe that Theorem \ref{lfv} should not depend on the point $p$. However, the inequalities of the present subsection are more complicated when the partitions have nonempty $l$-cores. Since Theorem \ref{lfv} is strong enough for our purposes, we do not pursue this here.
\end{rem}

\section{Fusion operator and quantum toroidal algebra}

In \cite{OkSm}, Okounkov and Smirnov constructed a certain quantum group $\moalgebra$ which acts on the equivariant $K$-theory of quiver varieties. The main result of \cite{OkSm} identifies the quantum difference operator with an operator inside (a certain completion) of $\moalgebra$. In \cite{zhu2}, it was proven that $\moalgebra \cong \qta$, where $\qta$ denotes the quantum toroidal algebra for $\mathfrak{g}=\mathfrak{sl}_{l}$. So we can apply known results about $\qta$ to the enumerative geometry of quiver varieties. Our goal in this section is Proposition \ref{fusion}, and we introduce the minimum amount needed to get there.

\subsection{Quantum toroidal algebra}\label{sec: qta}
We recall a few basic facts about $\qta$, the quantum toroidal algebra of $\mathfrak{g}=\mathfrak{sl}_{l}$, following \cite{Tsym} and \cite{Wen}. Inside $\qta$, there are two copies of the quantum affine algebra $\uqgh$ which are denoted by $h,v:\uqgh \hookrightarrow \qta$ and referred to as the ``horizontal" and ``vertical" subalgebras.

There are three representations of $\qta$ that we will consider. The first is the ``geometric representation"
\[
\rho_{g}:\qta \to \text{End}(\fock), \quad \fock:=\bigoplus_{\dv \in \mathbb{Z}_{\geq 0}^{l}} K_{\bT}(\qv(\dv,\delta_{0}))_{loc} 
\]
which was constructed in \cite{SV}, see also \cite{Neg}. The $K$-theory classes of skyscraper sheaves $[I_{\lambda}]$ of fixed points give a basis of this representation.

The quantum group $\moalgebra$ of \cite{OkSm} is defined as follows. Using stable envelopes, \cite{OkSm} defines a collection of $R$-matrices acting on $\fock \otimes \fock$. Using the FRT procedure of \cite{FRT}, $\moalgebra$ is defined as the algebra generated by all matrix coefficients of these $R$-matrices. It is known that $\moalgebra$ is generated by matrix coefficients of the so-called wall $R$-matrices which we will denote by $\rmat^{\text{OS}}_{w}$ for $w \in\mathbb{Q}^{l}$. 


In particular, the $R$-matrix of the horizontal subalgebra acts on $\fock \otimes \fock$. We write $\rmat$ for the universal quasi $R$-matrix of $\uqgh$.

\begin{thm}[\cite{zhu2} Theorem 7.2]\label{zhuR}
  $\rmat^{\text{OS}}_{0}=(\rho_{g} \otimes \rho_{g})(h(\rmat))$.
\end{thm}

\begin{rem}
To be clear, \cite{zhu2} proves a much stronger result. Nevertheless, we only require information about the $R$ matrix for wall $0$.
\end{rem}

The second $\qta$ representation that we consider is the ``fermionic Fock" representation $\rho_{f}: \qta \to \text{End}(\ferfock)$, where $\ferfock$ is the vector space over $\mathbb{Q}(t_1,t_2)$ with basis given by partitions $|\lambda\rangle$. For the definition of the $\qta$ action, see \cite{Wen} section 3.5.

Third, there is the ``bosonic" representation \cite{saito} (also called the vertex representation)
\[
\rho_{b}: \qta \to \text{End}(\bos), \quad \bos:=\Lambda_{t_{1},t_{2}}^{(l)}\otimes_{\mathbb{Q}} \mathbb{Q}[Q]
\]
which has a basis given by wreath Macdonald polynomials $H_{\lambda}$.

Tsymbaliuk proved that there is an isomorphism $T:\ferfock \xrightarrow{\sim} \bos$ which intertwines $\rho_{f}$ with $\rho_{b} \circ \omega$, where $\omega$ is the so-called Miki automorphism of $\qta$ \cite{Tsym}. Wen proved that $T(|\lambda\rangle)=a_{\lambda} H_{\lambda}$ for some $a_{\lambda}\in \mathbb{Q}(t_1,t_2)$ \cite{Wen}. The scalars $a_{\lambda}$ are difficult to compute, and a formula for them is not known \cite{Wen2}.

On the other hand, there exists $b_{\lambda}\in \mathbb{Q}(t_1,t_2)$ such that the map $\fock \to \ferfock$ induced by $[I_{\lambda}]\mapsto b_{\lambda}|\lambda\rangle$ is an isomorphism of $\qta$ representations, see \cite{Nagao}. Combinatorial formulas for $b_{\lambda}$ are known.

\begin{conj}\label{conj: scalars}
    We have the equality $a_{\lambda}=b_{\lambda}^{-1}$.
\end{conj}

\noindent\textbf{Assumption:} We henceforth assume Conjecture \ref{conj: scalars} holds. This implies that the isomorphism $\fock \xrightarrow{\sim} \bos$ induced by $[I_{\lambda}] \to H_{\lambda}$ intertwines $\rho_{g}$ with $\rho_{b} \circ \omega$.

\subsection{Khoroshkin-Tolstoy formula}\label{KTformula}

By the Khoroshkin-Tolstoy formula, see \cite{KT,KT2} and also \cite{zhu1} section 2.8, the quasi $R$-matrix of $\uqgh$ can be written in the form
\begin{equation}\label{ktrmatrix}
\rmat= \rmat^{+} \rmat_{0} \rmat^{-}
\end{equation}

This factorization depends on a choice of normal ordering of the positive affine roots. Let $\Delta$ denote the affine roots. We denote the affine simple roots by $\alpha_{0},\ldots,\alpha_{l-1}$ and the imaginary root by $\delta$. Let $\Delta_{+}^{\text{re}}$ denote the real positive roots. Let $\Delta^{\text{fin}}$ be the finite root system with simple roots $\alpha_1,\ldots,\alpha_{l-1}$, and let $\Delta^{\text{fin}}_{+}$ denote the positive finite roots. We choose a normal ordering, denoted $\leqslant$, such that all of the positive roots of the finite root system precede the imaginary root. 

Then \cite{KT2} constructs a quantum Cartan-Weyl basis for $\uqgh$. It consists of elements $e_{\gamma} \in \uqgh$ for each $\gamma \in \Delta$. In terms of this basis, the factors $\rmat^{\pm}$ of the $R$-matrix are
\begin{equation}\label{eq: KT Rplus}
\rmat^{+}=\prod_{\substack{\gamma \in \Delta_{+}^{\text{re}} \\ \gamma<\delta}} \exp_{\hbar^{-(\gamma,\gamma)}}\left(a(\gamma) e_{\gamma} \otimes e_{-\gamma} \right), \qquad a(\gamma) \in \mathbb{C}(\hbar)
\end{equation}
and
\begin{equation}\label{eq: KT Rminus}
\rmat^{-}=\prod_{\substack{\gamma \in \Delta_{+}^{\text{re}} \\ \gamma>\delta}} \exp_{\hbar^{-(\gamma,\gamma)}}\left(a(\gamma) e_{\gamma} \otimes e_{-\gamma} \right), \qquad a(\gamma) \in \mathbb{C}(\hbar)
\end{equation}
where the products are ordered according to the choice of normal ordering and $\exp$ denotes
\[
\exp_{b}(a)=\sum_{n \geq 0} \frac{a^n}{(n)_{q}}, \quad (n)_{b}:=\frac{1-b^n}{1-b}
\]

By our assumption on the normal ordering, the roots $\gamma$ appearing in the product for $\rmat^{+}$ are of the form $\alpha+k \delta$ where $\alpha \in \Delta^{\text{fin}}_{+}$ and $k \geq 0$. Similarly, all the roots appearing in the product for $\rmat^{-}$ are of the form $\alpha_{0}+\alpha+k \delta$ where $\alpha \in \Delta^{\text{fin}}_{+}$ and $k \geq 0$.

\begin{lem}\label{lem: color support}

Let $\gamma \in \Delta^{\text{re}}_{+}$ and write $\gamma=\beta+k \delta$ where $k \geq 0$ and either $\beta\in \Delta^{\text{fin}}_{+}$ or $\beta=\alpha_0+\alpha$ for some $\alpha \in \Delta^{\text{fin}}_{+}$. Let $\operatorname{supp}(\beta)\subset \{0,\ldots,l-1\}$ be the indices of the simple roots appearing in $\beta$ with nonzero coefficient. Then $\rho_{g}(h(e_{\gamma}))$ (resp. $\rho_{g}(h(e_{-\gamma}))$) acts on $\fock$ in the basis $[I_{\lambda}]$ by adding (resp. removing) boxes of $\lambda$ with content in $\operatorname{supp}(\gamma)$.

\end{lem}
\begin{proof}

    It is proven in \cite{KT2} that elements of the quantum Cartan-Weyl basis of the form $e_{\alpha_i+k \delta}$ (resp. $e_{-(\alpha_i+k \delta)}$) match, up to constants, the coefficients of the $i$th generating current in the positive (resp. negative) half of Drinfeld's realization of $\uqgh$. From the definition of the Fock representation, see Theorem 3.18 of \cite{Wen}, $e_{\alpha_i+k \delta}$ (resp.  $e_{-(\alpha_i+k \delta)}$) thus acts by adding (resp. removing) boxes of content $i$.

Let $\gamma=\beta+k \delta \in \Delta^{\text{re}}_{+}$ be as in the statement of the lemma. The quantum Cartan-Weyl basis is constructed inductively. In particular, if $\beta$ is not a simple root, then $e_{\beta+k \delta}$ is a $\hbar$-twisted commutator of $e_{\beta'+k' \delta}$ and $e_{\beta''+k'' \delta}$ for some $k',k'' \geq 0$ and some roots $\beta',\beta''$ whose supports are contained strictly in $\operatorname{supp}(\beta)$. So the claim follows from the previous paragraph by induction on the size of the support.
\end{proof}

We need the following formula for the action of $h(\rmat_{0})$ in the tensor product of two geometric representations.

\begin{lem}\label{R0}
Under the identification $\fock \cong \bos$, we have
    $$
(\rho_{g} \otimes \rho_{g}) h(\rmat_{0}) = \exp \left(\sum_{i=0}^{l-1} \sum_{n \geq 1} (\hbar^{n/2} -\hbar^{-n/2}) p^{(i)}_{n} \otimes  \frac{\partial}{\partial p^{(i)}_{n}}  \right)
    $$
\end{lem}
\begin{proof}
    We follow the notation of \cite{Wen}, up to a transposition of tensor factors. By Proposition 3.4 of \cite{Wen},
$$
h(\rmat_{0}) =\exp\left(\sum_{i=0}^{l-1} \sum_{n \geq 1} \omega^{-1}(b_{i,-n}) \otimes \omega^{-1}(b_{i,n}^{\perp}) \right)
$$
where 
\begin{itemize}
    \item $b_{i,k}$ are certain elements of $\qta$ satisfying $\rho_{b}(b_{i,-k})=\frac{\qint{k}{\sqrt{\hbar}}}{k} p^{(i)}_{k}$ (here $p^{(i)}_{k}$ stands for the multiplication operator and $\qint{k}{\sqrt{\hbar}}=\frac{\hbar^{k/2}-\hbar^{-k/2}}{\hbar^{1/2}-\hbar^{-1/2}}$ is the quantum integer).
    \item $b_{i,k}^{\perp}$ satisfies
$$
\langle b_{i,k}^{\perp}, b_{j,-k'} \rangle = \delta_{k,k'} \delta_{i,j}
$$
under the pairing defined by
$$
\langle b_{i,k}, b_{j,-k'} \rangle = -\delta_{k,k'} \frac{\qint{k a_{i,j} }{\sqrt{\hbar}} }{k(\hbar^{1/2}-\hbar^{-1/2})}
$$
for $k >0$.
\item $\omega$ stands for the Miki automorphism of the quantum toroidal algebra, for which it is known that $\omega \circ h = v$ (see \cite{Miki1, Miki2, Wen}), and by Conjecture \ref{conj: scalars}, $\rho_{b} \circ \omega = \rho_{g}$.
\end{itemize}
It is easy to compute that $\rho_{b}(b_{i,k}^{\perp})=k(\hbar^{1/2}-\hbar^{-1/2}) \frac{\partial}{\partial p^{(i)}_{k}}$. So we obtain
\begin{align*}
    (\rho_{g} \otimes \rho_{g}) h(\rmat_{0})&= (\rho_{b} \otimes \rho_{b}) \circ (\omega \otimes \omega) h(\rmat_{0}) \\
    &=\exp\left(\sum_{i=0}^{l-1} \sum_{n \geq 1} \rho_{b}(b_{i,-n}) \otimes \rho_{b}(b_{i,n}^{\perp}) \right) \\
    &=\exp\left(\sum_{i=0}^{l-1} \sum_{n \geq 1} (\hbar^{n/2}-\hbar^{-n/2}) p^{(i)}_{n} \otimes \frac{\partial}{\partial p^{(i)}_{n}} \right)
\end{align*}
\end{proof}

\subsection{Matrix elements of the fusion operator}

We recall the following definition from \cite{OkSm}.

\begin{defn}
   Let $A\in \text{End}(\fock \otimes \fock)$ be an operator of the form $A=\bigoplus_{\alpha \in \mathbb{Z}^{l}} A_{\alpha}$ where for any $\dv$ and $\dv'$, $A_{\alpha}(K_{\bT}(\qv(\dv,\delta_{0})\times \qv(\dv',\delta_{0})))  \subset K_{\bT}(\qv(\dv+\alpha,\delta_{0})\times \qv(\dv'-\alpha,\delta_{0}))$. Then $A$ is upper triangular if $A_{\alpha} = 0$ unless $\sum_{i=0}^{l-1} \alpha_{i} \geq 0$. Furthermore, if $A_{0}=1$, then $A$ is strictly upper triangular.
\end{defn}

Let $F=\qv(\dv,\delta_{0})\times \qv(\dv',\delta_{0})$ so that $K_{\bT}(F)$ is a direct summand of $\fock \otimes \fock$. Recall that $\fock \cong \bos$ has a basis given by wreath Macdonald polynomials. For later use, we let $\fock_{\emptyset}$ be the subspace spanned by wreath Macdonald polynomials with empty $l$-core.

Let $\Omega$ be the $K_{\bT}(pt)$ linear operator on $\fock \otimes \fock$ which acts on $K_{\bT}(F)$ by multiplication by one fourth the codimension of $F$ in $\qv(\dv+\dv',2 \delta_{0})$. An explicit formula for $\Omega$ is given in section 2.3.7 of \cite{OkSm}. Let $Z_{(1)}$ be the operator which acts on $K_{\bT}(F)$ as multiplication by $\prod_{i=0}^{l-1} z_{i}^{\dv_{i}}$.

\begin{prop}[\cite{OkSm} Proposition 7]\label{fusiondef}
    There is a unique element $J(z) \in (\rho_{g} \otimes \rho_{g}) (h(\uqgh)\otimes h(\uqgh))$ such that $J(z)$ is strictly upper triangular and solves the ABRR equation:
    $$
(\rho_{g} \otimes \rho_{g})(h(\rmat)) Z^{-1}_{(1)} J(z) Z_{(1)}= \hbar^{-\Omega} J(z) \hbar^{\Omega} 
$$
\end{prop}

\begin{rem}
    The operator $J(z)$ is known as the ``fusion operator". It is defined in \cite{OkSm} using the wall $0$ $R$-matrix. Due to Theorem \ref{zhuR}, we can equivalently use the $R$-matrix of the horizontal subalgebra. 
\end{rem}

Our interest in $J(z)$ is due to Proposition \ref{capfactor} below. The general formula for the solution to these equations is complicated, see \cite{zhu1}. For our purposes, we will only need certain matrix elements of $J(z)$ for which the formulas are much simpler. We first consider matrix elements of $h(\rmat)$ on $\fock$ in the basis of torus fixed points.

\begin{prop}\label{Rmatrixelements}
If $\lambda$, $\mu$, and $\nu$ have empty $l$-core, then
    $$
\langle [I_{\nu}] \otimes 1| (\rho_{g} \otimes \rho_{g}) h(\rmat)| [I_{\lambda}] \otimes [I_{\mu}] \rangle =\langle [I_{\nu}] \otimes 1| (\rho_{g} \otimes \rho_{g}) h(\rmat_{0})| [I_{\lambda}] \otimes [I_{\mu}] \rangle
    $$

\end{prop}

\begin{proof}
    
    We apply the KT-factorization: $\rmat=\rmat^{+} \rmat_{0} \rmat^{-}$. Recall the discussion at the beginning of section \ref{KTformula} regarding our choice of normal order of the affine roots. By \eqref{eq: KT Rplus} and Lemma \ref{lem: color support}, in the fixed point basis of $\fock$, $(\rho_{g} \otimes \rho_{g})h(\rmat^{+})=1+\ldots$ where the dots represent terms that remove boxes of content other $0$ from the second tensorand. Since any nonempty partition has a box of content 0, we deduce that $\langle [I_{\nu}] \otimes 1|(\rho_{g} \otimes \rho_{g})h(\rmat^{+})| v \rangle=\langle [I_{\nu}] \otimes 1| v \rangle$ for any $v \in \fock \otimes \fock$. 

Again, by \eqref{eq: KT Rminus} and Lemma \ref{lem: color support}, we see that $(\rho_{g} \otimes \rho_{g})h(\rmat^{-})=1+\ldots$ where the dots stand for terms that remove boxes from the second tensorand but cannot remove an equal number of boxes of each color. Thus if $\lambda$ and $\mu$ have empty $l$-core, then $(\rho_{g} \otimes \rho_{g})h(\rmat^{-})( [I_{\lambda}] \otimes [I_{\mu}])=[I_{\lambda}] \otimes [I_{\mu}] + \ldots$, where the dots stand for a linear combination of classes $[I_{\gamma}] \otimes [I_{\delta}]$ where $\core_{l}(\delta)\neq \emptyset$.

    Finally, by Lemma \ref{R0}, $(\rho_{g} \otimes \rho_{g})h(\rmat^{0})$ preserves the $l$-core of each tensorand. The result follows.

\end{proof}

Recall the identification $\fock \cong \bos$ which maps $[I_{\lambda}]$ to $H_{\lambda}$.

\begin{prop}\label{fusion}
If $\lambda$, $\mu$, and $\nu$ have empty $l$-core, then
    $$
    \langle [I_{\nu}] \otimes 1 |J(z)| [I_{\lambda}] \otimes [I_{\mu}] \rangle=  \langle H_{\nu} \otimes 1 |J_{0}(z)| H_{\lambda} \otimes H_{\mu} \rangle 
    $$
    where
    $$
J_{0}(z)=\exp \left(\sum_{i=0}^{l-1} \sum_{n \geq 1} \frac{(\hbar^{n/2}-\hbar^{-n/2})}{1-(z_0\ldots z_{l-1})^{-n}} p^{(i)}_{n} \otimes  \frac{\partial}{\partial p^{(i)}_{n}}   \right)
    $$
\end{prop}
\begin{proof}

By Proposition \ref{Rmatrixelements}, the matrix elements in the left side of the statement of the proposition agree with those of the unique strictly upper triangular solution $J_{0}(z)$ of the equation 
\begin{equation}\label{ABRR0}
(\rho_{g} \otimes \rho_{g}) (h(\rmat_{0})) Z^{-1}_{(1)} J_{0}(z) Z_{(1)}= \hbar^{-\Omega} J_{0}(z) \hbar^{\Omega} 
\end{equation}
restricted to the subspace $\fock_{\emptyset} \otimes \fock_{\emptyset} \subset \fock \otimes \fock$.

From \cite{OkSm} formula (35), $\Omega|_{K_{\bT}(\qv(n_1) \times \qv(n_2))}= \frac{1}{2}(n_1+n_2)$.

To solve the previous equation, we use the identification $\fock \cong \bos$ and make the ansatz inspired by Lemma \ref{R0}:
$$
J_{0}(z)=\exp \left(\sum_{i=0}^{l-1} \sum_{n \geq 1} J^{(i)}_{n}(z) p^{(i)}_{n} \otimes \frac{\partial}{\partial p^{(i)}_{n}}   \right)
$$

Then \eqref{ABRR0} holds if
$$
J^{(i)}_{n}(z) (z_0\ldots z_{l-1})^{-n} + \hbar^{n/2}-\hbar^{-n/2}=J^{(i)}_{n}(z)
$$
Solving gives
$$
J^{(i)}_{n}(z)=\frac{(\hbar^{n/2}-\hbar^{-n/2})}{1-(z_0\ldots z_{l-1})^{-n}}
$$
which gives the result.

\end{proof}

Note that for $l=1$ we recover the fusion operator for $\hilb^n(\C^2)$, see section 7 of \cite{OkSm}.

\section{Quasimap counts and tensor product of quiver varieties}

\subsection{Limit of vertex function}

Let $\qv=\qv(\dv,\dw)$ be a $\hat{A}_{l-1}$ quiver variety. Split the framing as $\dw=u_1 \dw' + u_2 \dw''$ and let the corresponding subtorus of the framing torus be $\bA'$. We consider the limit of the vertex $\ver{\tau}_{\qv}$ of $\qv$ as $u_2 \to \infty$. It is proved in Lemma 7.3.5 of \cite{Ok17} that $\lim_{u_2 \to \infty} \ver{\tau}_{\qv}$ factorizes, up to a shift, as long as the limit of the descendant $\tau$ exists. Here we make the shift explicit.

Note that the ``tensor product of quiver varieties" also holds for the stack quotients:
$$
\stackqv(\dv,\dw)^{\bA'}= \bigsqcup_{\dv'+\dv''=\dv}\stackqv(\dv',\dw') \times \stackqv(\dv'',\dw'')
$$
We denote $F^{\text{st}}=\stackqv(\dv',\dw') \times \stackqv(\dv'',\dw'')$. Hence if $\tau \in \desc(\dv,\dw)$, then $\tau|_{F^{\text{st}}} \in \desc(\dv',\dw') \otimes \desc(\dv'',\dw'') \otimes \mathbb{Q}[u_1^{\pm 1},u_2^{\pm 1}]$.

\begin{thm}\label{limvertex}
In the notation above, let $F=\qv(\dv',\dw')\times \qv(\dv'',\dw'')=:\qv' \times \qv''$ be an $\bA'$ fixed component of $\qv$. If $\lim_{u_2 \to \infty} \tau|_{F^{\text{st}}}= \tau' \otimes \tau''$ where $\tau'\in \desc(\dv',\dw')$ and $\tau'' \in \desc(\dv'',\dw'')$, then
    $$\lim\limits_{u_2 \rightarrow \infty} \ver{\tau}_{\qv}(z)|_{F}= \ver{\tau'}_{\qv'}(z')\otimes \ver{\tau''}_{\qv''}(z'')$$
    where
    $z'_{i}=z_{i} q^{-\dv''_{i}+\dv''_{i+1}}(-\hbar^{1/2})^{\dw''_{i}+\dv''_{i-1}-2\dv''_{i}+\dv''_{i+1}}$ and $z_{i}''=z_{i} q^{-\dw_{i}'+\dv'_{i}-\dv'_{i+1}}(-\hbar^{1/2})^{-\dw'_{i}-\dv'_{i-1}+2\dv'_{i}-\dv'_{i+1}}$.
\end{thm}
\begin{proof}
   This is Lemma 7.3.5 of \cite{Ok17} spelled out for cyclic type $A$ quiver varieties. In general, the powers of $q$ arises from the polarization, and the powers of $-\hbar^{1/2}$ are the components of the vector $\dw''+C \dv''$ and $-\dw'- C \dv'$, where $C$ is the Cartan matrix of the quiver.
\end{proof}

A similar statement holds for the limit $u_1 \to 0$. In fact, the difference between $u_1 \to 0$ and $u_2 \to \infty$ is only present in the descendant $\tau$.

\begin{cor}\label{limvertex2}
    If, in particular, $\qv=\qv((n,n,\ldots,n),2 \delta_{0})$, $\qv'=\qv(n)$, and $\qv''=\qv(0)=\{pt\}$, then 
    $$
    \lim_{u_2 \to \infty} \ver{ \tau_{i}}_{\qv}(z)|_{F}=\ver{\tau_{i}}_{\qv'}(z)|_{z_{0}=z_{0}(-\hbar^{1/2})} \otimes 1
    $$
\end{cor}

\subsection{Limit of capping operator}

Consider $\qv:=\qv(\dw)$ and split the framing as $\dw=u_1 \dw'+ u_2 \dw''$. Let the corresponding subtorus of the framing torus be $\bA'$. Then $\qv(\dw)^{\bA'}=\qv(\dw') \times \qv(\dw'')=:\qv'\times\qv''$. Let $\stab:K_{\bT}(\qv(\dw)^{\bA'}) \to K_{\bT}(\qv(\dw))$ be the stable envelope \cite{Ok17} with the chamber such that $a=u_1/u_2$ is an attracting weight,
polarization given by   $$
        T^{1/2}=\hbar\sum_{i=0}^{l-1}\Hom(\tb_{i+1},\tb_{i})+\hbar\Hom(\tb_0,\mathcal{W}_{0})-\hbar\sum_{i=0}^{l-1}\Hom(\tb_i,\tb_i)
    $$ and the slope given by a small negative multiple of an ample line bundle. Let $\text{Res}: K_{\bT}(\qv(\dw)) \to K_{\bT}(\qv(\dw)^{\bA'})$ be the restriction map.

Smirnov proves the following.
\begin{thm}[\cite{Sm16}]
Let $\Psi_{\qv}$, $\Psi_{\qv'}$, and $\Psi_{\qv''}$ be the capping operators for $\qv$, $\qv'$, and $\qv''$. Then
$$
\lim_{a \to 0} (\text{Res} \circ \stab)^{-1} \Psi_{\qv} (\text{Res} \circ \stab) = J(z) \Psi_{\qv'}(z') \otimes \Psi_{\qv''}(z'')
$$
where the fixed-component dependent shift of the K\"ahler parameters is $z_{i}'=z_{i} (-\hbar^{1/2})^{\dw_{i}''+\dv_{i-1}''-2\dv_{i}''+\dv_{i+1}''}$ and $z_{i}''=z_{i} (-\hbar^{1/2})^{-\dw_{i}'-\dv_{i-1}'+2\dv_{i}'-\dv'_{i+1}}$ and $J(z)$ is the operator from Proposition \ref{fusiondef}.
\end{thm}

\begin{rem}
    Smirnov's paper \cite{Sm16} considers only the Jordan quiver. Nevertheless, his techniques apply equally well to all quivers, with the only needed modification being the fixed-component dependent shift in the statement of the theorem.
\end{rem}

Let $\Delta_{\dw',\dw''}\in \text{End}\left(K_{\bT}(\qv(\dw)^{\bA'})\right)$ be the operator of multiplication by the diagonal of the stable envelope above. More precisely, if $\gamma \in F:=K_{\bT}(\qv(\dv',\dw')) \otimes K_{\bT}(\qv(\dv'',\dw''))$, then 
$$
\Delta_{\dw',\dw''}(\gamma)=(-1)^{\text{rank}(T^{1/2}_{>0})} \sqrt{\frac{\det N_{-}}{\det T^{1/2}_{\neq 0}}} \extpow\left(N_{-}^{\vee}\right)
$$
where $N_{-}$ is the repelling part of the normal bundle to $F$ in $\qv$ and 
$$
T^{1/2}\qv |_{F}= \underbrace{T^{1/2}_{<0} \oplus T^{1/2}_{>0}}_{T^{1/2}_{\neq 0}} \oplus T^{1/2}_{0}
$$
is the decomposition into repelling, fixed, and attracting parts. Notice that $\lim_{u_2 \to 0} \extpow(N_{-}^{\vee})=1$.

\begin{lem}
    $$
\lim_{a \to 0} \Delta_{\dw',\dw''}^{-1} \circ \text{Res} \circ \stab = 1
    $$
\end{lem}
\begin{proof}
    This follows immediately from the three axioms for stable envelopes, see \cite{OkSm}.
\end{proof}

Putting the last two results together, we deduce the following.
\begin{prop}\label{capfactor}
    $$
    \lim_{a \to 0} \Delta_{\dw',\dw''}^{-1} \Psi_{\qv} \Delta_{\dw',\dw''}= J(z) \Psi_{\qv'}(z') \otimes \Psi_{\qv''}(z'')
    $$
\end{prop}

We also need the following lemma. 

\begin{lem}\label{limlb}
Let $\mathscr{L}_{0}^{(2)}$ be the operator of tensor multiplication by the line bundle $\det \tb_{0}$ on $\qv(2 \delta_{0})$. The component of
    $$
\lim_{u_2 \to \infty} \mathscr{L}_{0} ^{(2)} \Delta_{\delta_{0},\delta_{0}}^{-1}
    $$
    acting on $K_{\bT}(F)$ where $F=\qv(\dv',\delta_{0})\times \qv(\dv'',\delta_{0})\subset \qv(n,2 \delta_{0})$ is nonzero if and only if $\dv'=(n_1,n_1,\ldots,n_1)$ and $\dv''=(n_2,n_2,\ldots,n_2)$, in which case it is the operator of multiplication by
    $$
u_1^{n_2} (-\hbar^{1/2})^{n_2} (\mathscr{L}_{0} \otimes 1)
    $$
\end{lem}

\begin{proof}
   At a fixed component $F$, all the tautological bundles on decompose as $\tb_{i}=\tb_{i}'\oplus \tb_{i}''$  and $\mathcal{W}_{i}= \mathcal{W}_{i}' \oplus \mathcal{W}_{i}''$ where $u_1$ (resp. $u_2$) scales the first (resp. second) summands. So, by our choice of polarization, we have
    \begin{multline*}
T^{1/2}_{\neq 0} |_{F}=\hbar \sum_{i=0}^{l=1} \left( \Hom(\tb'_{i+1},\tb''_{i}) + \Hom(\tb_{i+1}'', \tb_{i}') \right)  \\ + \hbar \Hom(\tb'_{0},\mathcal{W}_{0}'') + \hbar \Hom(\tb''_{0},\mathcal{W}'_{0}) -  \hbar \sum_{i=0}^{l-1} \left( \Hom(\tb'_{i},\tb''_{i})+  \Hom(\tb''_{i},\tb'_{i}) \right)
    \end{multline*}
   Our choice of chamber gives that $u_2/u_1$ is a repelling weight. Thus
   \begin{multline*}
       N_{-}|_{F}=\hbar\sum_{i=0}^{l-1} \left( \Hom(\tb_{i}',\tb_{i+1}'') + \Hom(\tb'_{i+1},\tb''_{i}) \right) \\
       + \hbar\Hom(\mathcal{W}_{0}',\tb''_{0}) + \hbar \Hom(\tb_{0}',\mathcal{W}_{0}'') - \hbar\sum_{i=0}^{l-1}\left( \Hom(\tb_{i}',\tb_{i}'') + \Hom(\tb_{i}',\tb_{i}'') \right)
\end{multline*}
and
\begin{equation*}
    T^{1/2}_{>0} |_{F}=\sum_{i=0}^{l=1} \hbar\Hom(\tb_{i}'', \tb_{i+1}')  + \hbar \Hom(\tb''_{0},\mathcal{W}'_{0}) -  \hbar \sum_{i=0}^{l-1}  \Hom(\tb''_{i},\tb'_{i}) 
\end{equation*}
So $\text{rank}(T^{1/2}_{>0}|_{F})=\dv_{0}''+\sum_{i=0}^{l-1} (\dv_{i+1}'-\dv_{i}') \dv_{i}''$.
Using the identity $\det(\Hom(\tb,\mathcal{W}))=\det(\mathcal{W})^{\text{rank}(\tb)}\det(\mathscr{V})^{-\text{rank}(\mathcal{W})}$ we have
\begin{align*}
   \left( \det(\tb_{0})  \sqrt{\frac{\det T^{1/2}_{\neq 0}}{\det N_{-}}} \right)\bigg |_{F}=\det(\tb'_{0}) u_{1}^{\dv''_{0}} (\hbar^{1/2})^{\dv''_{0}+ \sum_{i=0}^{l-1} (\dv'_{i} \dv''_{i+1} - \dv'_{i} \dv''_{i})}\prod_{i=0}^{l-1} \det(\tb_{i}'')^{\dv'_{i}-\dv'_{i-1}} \det (\tb'_{i})^{\dv''_{i+1}-\dv''_{i}}
\end{align*}

Since $\dv'+\dv''=(n,n,\ldots,n)$, the total power of $u_2$ in the above expression is 
$$
\sum_{i=0}^{l-1} \dv''_{i}(\dv'_{i}-\dv'_{i-1})=\sum_{i=0}^{l-1} (n-\dv'_{i})(\dv'_{i}-\dv'_{i-1}) = -\frac{1}{2} \sum_{i=0}^{l-1} (\dv_{i-1}'-\dv_{i}')^{2}
$$ 
which is always nonpositive and is zero if and only if $F \subset \qv^{0} \times \qv^{0}$. Since $\lim_{u_2 \to \infty} \extpow(N_{-}^{\vee}|_{F})=1$, this proves the first part of the lemma.

For the second part, suppose $\dv'=(n_1,n_1,\ldots,n_1)$ and $\dv''=(n_2,n_2,\ldots,n_2)$. Then the above calculation shows that
$$
\lim_{u_2 \to \infty} \Delta^{-1} \mathscr{L}_{0}^{(2)}=  u_1^{n_2} (-\hbar^{1/2})^{n_2}\mathscr{L}_{0} \otimes 1
$$
\end{proof}

\section{Derivation of main formula}

We now have all the ingredients needed to prove the main theorem. Recall the quantities $N_{\lambda}$ and $C^{(i)}_{m,n}$ from section \ref{symsec}.

\begin{thm}\label{mainthm}
We have
     \begin{multline*}
\sum_{\substack{\lambda \\ \core_{l}(\lambda)=\emptyset}} N_{\lambda}^{-1} H_{\lambda} \cver{\tau_{0}}_{\lambda}  = \\
\exp\left(\sum_{i=0}^{l-1}\sum_{n \geq 1} \frac{C^{(i)}_{0,n} p^{(i)}_{n}}{n (1-t_{1}^{nl})(1-t_{2}^{nl})} \left(1-\frac{(-\hbar^{1/2} u)^{n} - (\hbar u z_{0} \ldots z_{l-1})^{n}}{(\hbar^{1/2})^{n}-(-z_{0} \ldots z_{l-1})^n} \right)   \right)
\end{multline*}
\end{thm}

\begin{proof}    

We first set up the notations for the proof. Let $\qv^{(2)}_{\emptyset}=\bigsqcup_{n}\qv((n,n,\ldots,n),2 \delta_{0})$ and let $\bA$ be the (rank two) framing torus, which has coordinates $u_1$ and $u_2$. The superscript $(2)$ refers to the size of the framing, while the subscript $\emptyset$ refers to the fact that the $K$-theory of $\qv^{(2)}_{\emptyset}$ has a basis given by wreath Macdonald polynomials with \textit{empty l-core}. Let $\tau^{(2)}_{0}$ denote the descendant from Definition \ref{descdef} on $\qv^{(2)}_{\emptyset}$ corresponding to the $0$th tautological bundle. Let $\ver{\tau^{(2)}_{0}}$ denote the corresponding descendant vertex function. Let $\Psi^{(2)}$ denote the capping operator for $\qv^{(2)}_{\emptyset}$. Let $\mathscr{L}_{0}^{(2)}$ denote the operator of tensor multiplication in $K$-theory by the $0$th tautological line bundle of $\qv^{(2)}_{\emptyset}$.

Similarly, we have $\qv^{(1)}_{\emptyset}=\bigsqcup_{n} \qv((n,n,\ldots,n), \delta_{0})$, $\tau^{(1)}_{0}$, $\mathscr{L}_{0}^{(1)}$, and $\cver{\tau^{(1)}_{0}}$. This last object is the capped descendant vertex appearing in the left hand side of Theorem \ref{mainthm}.

From the inclusions, $\qv^{(1)}_{\emptyset}\times \qv^{(1)}_{\emptyset} \hookrightarrow (\qv^{(2)}_{\emptyset})^{\bA} \hookrightarrow \qv^{(2)}_{\emptyset}$ we have the restriction maps
$$
\begin{tikzcd}
    K_{\bT}\left(\qv^{(2)}_{\emptyset}\right) \arrow{r}{\text{Res}}  \arrow[bend right=30]{rr}{\text{Res}_{\emptyset}} & K_{\bT}\left((\qv^{(2)}_{\emptyset})^{\bA}\right) \arrow{r} & K_{\bT}\left(\qv^{(1)}_{\emptyset} \times \qv^{(1)}_{\emptyset}\right) 
\end{tikzcd}
$$

By Theorem \ref{VPsiV} and Theorem \ref{lfv}, we have
\begin{equation}\label{lfveq}
\tau_{0}^{(2)}= \Psi^{(2)} \ver{\tau^{(2)}_{0}}
\end{equation}
Applying $\text{Res}$ to \eqref{lfveq} and multiplying by the class $\Delta:=\Delta_{\delta_{0},\delta_{0}}$ from Proposition \ref{capfactor} gives

\begin{equation}\label{lfveq2}
    \Delta^{-1} \text{Res } \tau_{0}^{(2)}=  \left(\Delta^{-1} \Psi^{(2)} \Delta \right) \left(\Delta^{-1}  \ver{\tau^{(2)}_{0}}\right)
\end{equation}
Since the objects in the right hand side are defined by localization, we will omit the $\text{Res }$ from the right hand side. Next we take the limit of \eqref{lfveq2} as $u_2 \to \infty$. We consider each piece of \eqref{lfveq2} separately.

\begin{enumerate}
    \item First consider $\Delta^{-1} \text{Res } \tau^{(2)}_{0}$. We have
    \begin{align}
    \lim_{u_2 \to \infty} \Delta^{-1} \text{Res } \tau^{(2)}_{0} &=\lim_{u_2 \to \infty} \Delta^{-1} \mathscr{L}_{0}^{(2)} \left(\mathscr{L}_{0}^{(2)}\right)^{-1} \text{Res } \tau^{(2)}_{0} \nonumber \\
    &=\lim_{u_2 \to \infty} \Delta^{-1} \mathscr{L}_{0}^{(2)} \left( \mathscr{L}_{0}^{(2)} \right)^{-1} \text{Res}_{\emptyset} \, \tau^{(2)}_0 \label{limlhs}
     \end{align}
    where the second equality follows from Lemma \ref{limlb}. Under the identification with symmetric functions of Proposition \ref{Ksym}, we have
$$
\text{Res}_{\emptyset} \, \tau^{(2)}_0=\sum_{\substack{\lambda, \mu \in \parts \\ \core_{l}(\lambda)= \emptyset \\ \core_{l}(\mu)= \emptyset}} N_{\lambda}^{-1} N_{\mu}^{-1} H_{\lambda}[x] H_{\mu}[y] \prod_{\substack{\square=(a,b) \in \lambda \sqcup \mu \\ a-b \equiv 0 \mod l}}(1+u_{\square} t_{1}^{1-b} t_{2}^{1-a})
$$
where $u_{\square}=u_{1}$ if $\square \in \lambda$ and $u_{\square}=u_{2}$ if $\square \in \mu$. By Lemma \ref{limlb}, the explicit formula for $\mathscr{L}_{0}^{(2)}$, and the homogeneity of wreath Macdonald polynomials, \eqref{limlhs} is equal to
\begin{align*}
 &\sum_{\substack{\lambda, \mu \in \parts \\ \core_{l}(\lambda)= \emptyset \\ \core_{l}(\mu)= \emptyset}} N_{\lambda}^{-1} N_{\mu}^{-1} H_{\lambda}[x] (-u_1 \hbar^{1/2})^{|\mu|/l}H_{\mu}[y] \prod_{\substack{\square=(a,b) \in \lambda \\ a-b \equiv 0 \mod l}}(1+u_1 t_{1}^{1-b} t_{2}^{1-a}) \\ 
     &=\sum_{\substack{\lambda, \mu \in \parts \\ \core_{l}(\lambda)= \emptyset \\ \core_{l}(\mu)= \emptyset}} N_{\lambda}^{-1} N_{\mu}^{-1} H_{\lambda}[x] H_{\mu}[-u_1 \hbar^{1/2} y] \prod_{\substack{\square=(a,b) \in \lambda \\ a-b \equiv 0 \mod l}}(1+u_1 t_{1}^{1-b} t_{2}^{1-a})
\end{align*}

    \item Second, we consider $\Delta^{-1} \Psi^{(2)} \Delta$. By Proposition \ref{capfactor}, 
$$
\lim_{u_2 \to \infty} \Delta^{-1} \Psi^{(2)} \Delta =  J(z) \Psi^{(1)}(z') \otimes \Psi^{(1)}(z'')
$$

\item The component of $\lim_{u_2 \to \infty}\Delta^{-1} \ver{\tau^{(2)}_{0}} $ for the summand $K_{\bT}\left(\qv^{(1)}_{\emptyset} \times \{pt\}\right)$ of $K_{\bT}(\qv^{(2)}_{\emptyset})$ is $\ver{\tau^{(1)}_{0}}(z') \otimes 1$. This follows from Corollary \ref{limvertex2} and Lemma \ref{limlb}.

\end{enumerate}

Applying the last three computations to \eqref{lfveq2}, we get
\begin{equation}\label{lfveq3}
\Psi^{(1)}(z') \otimes \Psi^{(1)}(z'') \lim_{u_2 \to \infty} \Delta^{-1} \ver{\tau^{(2)}_{0}}= J(z)^{-1} \lim_{u_2 \to \infty} \Delta^{-1} \text{Res}_{\emptyset} \, \tau^{(2)}_{0}
\end{equation}
The component of the left side of \eqref{lfveq3} for the summand $K_{\bT}\left(\qv^{(1)}_{\emptyset} \times \{pt\}\right)$ (in the language of symmetric functions, this means the degree zero part of the second tensorand) is
$$
\left(\Psi^{(1)}(z') \otimes 1 \right)\ver{\tau^{(1)}_{0}}(z') \otimes 1 = \cver{\tau^{(1)}_{0}}(z') \otimes 1
$$
Considering this in the ring of symmetric functions and doing the same for the right hand side of \eqref{lfveq3} gives 
\begin{multline}
\sum_{\substack{\lambda \\ \core_{l}(\lambda)=\emptyset}} N_{\lambda}^{-1} H_{\lambda}[x] \cver{\tau_{0}}_{\lambda}|_{z_{0}=z_{0}(-\hbar^{1/2})} \\
=\left(J(z)^{-1} \sum_{\substack{\lambda, \mu \\ \core_{l}(\lambda)=\emptyset \\ \core_{l}(\mu)=\emptyset}} N_{\lambda}^{-1} N_{\mu}^{-1} H_{\lambda}[x] H_{\mu}[-u_1\hbar^{1/2} y]\tau_{0}|_{\lambda} \right)\bigg|_{y=0} \label{Jinv}
\end{multline}

From Theorem \ref{fusion}, $J(z)^{-1}|_{K_{\bT}(\qv^{(1)}_{\emptyset}) \otimes K_{\bT}(\qv^{(1)}_{\emptyset})}$, followed by $y=0$, acts by $p^{(i)}_{n}[x] \mapsto p^{(i)}_{n}[x]$ and $p^{(i)}_{n}[y]\mapsto  \frac{(\hbar^{-n/2}-\hbar^{n/2})}{1-(z_0\ldots z_{l-1})^{-n}} p^{(i)}_{n}[x]$. Combining this with Lemma \ref{classicalformula} shows that the right side of \eqref{Jinv} is equal to

\begin{multline*}
\left( J(z)^{-1} \exp \left(\sum_{i=0}^{l-1} \sum_{n \geq 1} \frac{C^{(i)}_{0,n}}{n(1-t_{1}^{nl})(1-t_2^{nl})} \left((1-(-u)^{n}) p^{(i)}_{n}[x] + (-u \hbar^{1/2})^{n}p^{(i)}_{n}[y] \right)\right)\right)\Bigg|_{y=0} \\
=\exp \left(\sum_{i=0}^{l-1} \sum_{n \geq 1} \frac{C^{(i)}_{0,n} p^{(i)}_{n}[x]}{n(1-t_{1}^{nl})(1-t_2^{nl})}\left( (1-(-u)^{n})  + \frac{(-u)^{n}(1-\hbar^{n})}{1-(z_0\ldots z_{l-1})^{-n}} \right) \right) 
\end{multline*}

Shifting $z_{0}\mapsto z_{0}(-\hbar^{-1/2})$ and doing elementary algebra gives
$$
\sum_{\substack{\lambda \\ \core_{l}(\lambda)=\emptyset}} N_{\lambda}^{-1} H_{\lambda}[x] \cver{\tau_{0}}_{\lambda} = \exp \left(\sum_{i=0}^{l-1} \sum_{n \geq 1} \frac{C^{(i)}_{0,n} p^{(i)}_{n}[x]}{n(1-t_{1}^{nl})(1-t_2^{nl})}\left( 1-\frac{(-\hbar^{1/2} u)^{n}-(\hbar u z_0\ldots z_{l-1})^{n}}{(\hbar^{1/2})^{n}-(-z_0\ldots z_{l-1})^{n}} \right) \right)
$$
the main formula.

\end{proof}

\printbibliography

\pagebreak

\noindent Jeffrey Ayers \\
University of North Carolina \\
Chapel Hill, NC \\
jeff97@live.unc.edu

\vspace{0.5cm}

\noindent Hunter Dinkins \\
Massachusetts Institute of Technology\\
Cambridge, MA \\
hunte864@mit.edu

\end{document}